\documentclass[12pt,leqno]{amsart}
\usepackage{amsmath,amsthm}
\usepackage{amsfonts}
 \newtheorem{res}{Result}[section]
 \newtheorem{theorem}[res]{Theorem}
 \newtheorem{example}[res]{Example}
 \newtheorem{counterexample}[res]{Counterexample}
 \newtheorem{remark}[res]{Remark}
 \newtheorem{prop}[res]{Proposition}
 \newtheorem{lem}[res]{Lemma}
 \newtheorem{cor}[res]{Corollary}
 \newtheorem{defi}[res]{Definition}

\numberwithin{equation}{section}

\def\th{\theta}
\def\m{\mu}
\def\l{\lambda}
\def\b{b{\mathcal F}^+}
\def\B{\mathcal B}
\def\K{\mathcal K}
\def\C{\mathcal C}
\def\bM{\mathbb M}
\def\tX{\tilde X}
\def\tF{{\tilde{\mathcal F}}}
\def\defto{\buildrel def\over =}

\def\L{\mathcal L}

\def\calN{\mathcal N}
\def\trade{{K}}
\def\K{\trade}

\def\N{\mathbb N}
\def\E{\mathbb E}
\def\A{\mathcal A}
\def\Q{\mathcal Q}
\def\Om{\Omega}
\def\om{\omega}
\def\tOmega{{\tilde \Omega}}
\def\tOm{\tOmega}
\def\hOm{{\hat\Om}}
\def\tomega{{\tilde \omega}}
\def\tom{\tomega}
\def\barZ{{\bar Z}}
\def\al{\alpha}
\def\be{\beta}
\def\calN{\mathcal N}
\def\1{\mathbf 1}
\def\F{\mathcal F}
\def\hF{\hat{\F}}

\def\R{\mathbb R}
\def\P{\mathbb P}
\def\bP{\P}
\def\tP{{\tilde{\bP}}}
\def\hP{\hat{\bP}}

\def\bbQ{\mathbb Q}
\def\bQ{\bbQ}
\def\vare{\varepsilon}
\def\to{\rightarrow}

\def\Linf{{\mathcal L}^{\infty}}

\def\k{\kappa}

\def\E{\mathbb E}
\def\A{\mathcal A}
\def\Q{\mathcal Q}
\def\Om{\Omega}
\def\om{\omega}
\def\al{\alpha}
\def\be{\beta}
\def\1{\mathbf 1}
\def\G{\mathcal F}
\def\F{\mathcal F}
\def\Lplus{\L^{\infty}_+}
\def\H{\mathcal H}

\def\trho{\tilde \rho}
\def\1{\mathbf 1}
\def\hatv{\hat v}
\def\hatX{\hat X}
\def\endpf{\phantom{p}\hfill$\square$}
\def\half{\frac{1}{2}}

\begin{document}
\bibliographystyle{plain}
\title[On representing claims]{On representing claims
\\for coherent risk measures}
\date{}
\author{SAUL JACKA}
\address{Dept. of Statistics, University of Warwick, Coventry CV4 7AL, UK}
\email{s.d.jacka@warwick.ac.uk}
\author{ABDELKAREM BERKAOUI}
\address{College of Sciences, Al-Imam
Mohammed Ibn Saud Islamic University, P.O. Box 84880, Riyadh 11681,
Saudi Arabia.}
\email{berkaoui@yahoo.fr}
\begin{abstract}
We consider the problem of representing claims for coherent risk
measures. For this purpose we introduce the concept of (weak and
strong) time-consistency with respect to a portfolio of assets,
generalizing the one defined in Delbaen \cite{delbaen2}.

In a similar way we extend the notion of m-stability, by introducing
weak and strong versions. We then prove that the two concepts of m-
stability and time-consistency are still
equivalent, thus giving necessary and sufficient conditions for a
coherent risk measure to be represented by a market with
proportional transaction costs. We go on to deduce that, under a
separability assumption, any coherent risk measure is strongly
time-consistent with respect to a suitably chosen countable
portfolio, and show the converse: that any market with proportional
transaction costs is equivalent to a market priced by a coherent
risk measure, essentially establishing the equivalence of the two
concepts.
\end{abstract}

\thanks{The authors are grateful for many fruitful discussions with Jon
Warren on the topics of this paper.}

\thanks{This work was initiated while the second author was a Research
Fellow at Warwick University.}
\thanks{This research was partially supported by the grant \lq Distributed Risk Management'
in the Quantitative Finance initiative funded by EPSRC and the
Institute and Faculty of Actuaries}

\maketitle

\centerline{{\today}}


\section{Introduction}
The biggest practical success of Mathematical Finance to date is in explaining
how to
hedge against contingent claims (and thus how to price them uniquely) in the
context of a complete and frictionless market.

Two relatively recent developments in Mathematical Finance
are the introduction of the
concept of coherent risk measure and work on trading with (proportional)
transaction costs. Both of these developments seek to deal with deviations
from the idealised situation decribed above.

Coherent risk measures were first introduced by
Artzner, Delbaen, Eber and Heath \cite{ADEH}, in order to give a
broad axiomatic definition for monetary measures of risk.

In their fundamental theorem, Artzner {\em et al.} showed that such a
coherent risk measure can be represented as the supremum of
expectation over a set of test probabilities. Thus the setup includes
superhedging under the class of all EMMs (in an incomplete, frictionless
market).

Recent work on trading with transaction costs by Kabanov, Stricker,
Rasonyi, Jouini, Kallal, Delbaen, Valkeila and  Schachermayer,
amongst others (\cite{kabanov}, \cite{kabanov2}, \cite{JK},
\cite{DKV}, \cite{schacher}), lead to a necessary and sufficient
condition for the closure of the set of claims attainable for zero
endowment to be arbitrage-free (Theorem 1.2 of \cite{JBW}) and a characterisation
of the \lq dual' cone of pricing measures (consistent price
processes) (\cite{schacher}).

In this paper, we consider a coherent risk measure as a pricing
mechanism: in other words we assume that an economic agent is making a
market in (or at least reserving for) risk according to a coherent risk measure,
$\rho$ say.

So, we consider the
risk value of a financial claim as the basic price for
the associated contract. Unfortunately such a  pricing mechanism is not
closely linked to the notion of hedging, and so the price evolution from
trading time to maturity time is not well-defined. For example,
taking the obvious definition for $\rho_t$---the price of risk at time
$t$---it is not necessarily true that $\rho=\rho\circ\rho_t$ (see Delbaen
\cite{delbaen2}). Indeed, Delbaen has given a necessary and sufficient
condition for $\rho$ to be time-consistent in this way: the m-stability
property (\cite{delbaen2}), and this condition is easily violated. Notice that,
in the absence of m-stability,
reserving is not possible (without \lq new business strain'), since the time 0
price of (reserving for) the
time-$t$ reserve for a claim $X$ may (and sometimes will) be greater than the
time 0 reserve
for $X$.

Our preliminary results in this paper are as follows:
\begin{itemize}
\item[(1)]we introduce a generalisation of the concept of num\' eraire
suitable for the context of coherent risk measures (equation (\ref{num}))
and give a characterisation of such num\' eraires (Theorem \ref{numthm});

\item[(2)]we show (in equation (\ref{**}) and Lemma \ref{l2}) how to define a
$v$-denominated risk measure with the
same acceptance set as $\rho$, where $v$ is the final value of a positive claim
or of a different currency.
\end{itemize}
Then we pursue the idea of pricing using several
currencies/commodities/denominations.

If we do this, then the option of creating reserves in several
currencies becomes available. Moreover, the possibility of trading
between currencies or commodities in order to hedge a contingent claim
also appears.

Our main results are as follows:
\begin{itemize}
\item[(3)]in Theorem \ref{weakthm} we give a necessary and sufficient
condition (which generalises Delbaen's m-stability property)
for time-consistency with respect to a portfolio of assets (we term this weak
representation);

\item[(4)]in Theorems \ref{strongthm} and \ref{strongthm2}, we give two
necessary and sufficient conditions
(the first akin to Schachermayer's description
of the cone of consistent price processes) for
the attainability of all acceptable claims purely by trading in a portfolio of
assets;

\item[(5)]in Theorem \ref{separable} we show that, under a separability
condition, all acceptable claims may be attained by trading in a
fixed countable collection of assets;

\item[(6)]finally, we show, in Theorem \ref{reverse}, that every arbitrage-free
market
corresponding to trading with transaction costs in fact corresponds to
the representation of a coherent risk
measure using a set of commodities/num\' eraires .
\end{itemize}

\section{Preliminaries}

The paper is organized as follows: in section \ref{cond} we recall
properties of a conditional coherent risk measure. In section \ref{numsec}, we
consider a
one-period market, defined by a coherent risk measure, and define $\calN_0$, the
set of all num\'eraires in which we can trade in this market. Given
a num\'eraire $v\in \calN_0$, we define the $v$-denominated coherent
risk measure $\rho^v$. Remark that its value in cash (that is to say,
pieces of paper which pay 1 unit of account $\1$ at time $T$ i.e. Zero Coupon
Bonds),
given by
$\rho(\rho^v(X)v)$, may be different from $\rho(X)$. We discuss, in an
appendix,
the equivalence classes of num\'eraires, where such prices are the same.

Next, we
consider the general multi-period model and define $\calN$, the
set of all num\'eraires in which we can trade in every time period. In section
\ref{consist} we
introduce the two concepts of {\em time-consistency} and {\em m-stability} with
respect to a portfolio of assets in $\calN$. This new version of
time-consistency generalizes the one introduced in Delbaen
\cite{delbaen2}, and allows us not only to consider cash-flows but
also the possibility of investing in other assets. In order to show the
link between these two properties in section \ref{repsec}, we start with the
case where the portfolio of assets $V$ is finite and then consider
the cone $\A(V)$ of all portfolios in assets $V$, attainable from
non-positive endowment. Then we extend these results to the case
where the portfolio of assets is countable. The result in the finite case
is based on the results of section \ref{abscone}, where we consider a more
general cone $\B$ of portfolios attainable from non-positive
endowment. We will see that the notion of {\em decomposability} of the
cone $\B$, translated to the case $\B=\A(V)$, is equivalent to the
time-consistency property of the cone $\A$ with respect to $V$.

We assume that we are equipped with a filtered  probability space
$(\Om,\F,(\F_t)_{t=0,\ldots,T},\P)$, where $\F_0$ is not necessarily
trivial.

Now recall the setup from Schachermayer's paper \cite{schacher}: we
may trade in $d$ assets at times $0,\ldots,T$. We may burn any asset
and otherwise trades are given by a bid-ask process $\pi$ taking
values in $\R^{d\times d}$, with $\pi$ adapted to $(\F_t)_{t=0}^T$.
The bid-ask process gives the (time $t$) price for one unit of each
asset in terms of each other asset, so that
$$
\pi^{i,i}_t=1,\, \forall i,
$$
and $\pi^{i,j}_t$ is the (random) number of units of asset $i$ which
can be traded for one unit of asset $j$ at time $t$. We assume (with
Schachermayer) that we have \lq\lq netted out" any advantageous
trading opportunities, so that, for any $t$ and any
$i_0,\ldots,i_n$:
$$
\pi^{i_0,i_n}_t\leq \pi^{i_0,i_1}_t\ldots \pi^{i_{n-1},i_n}_t.
$$

The time $t$ trading cone, $\K_t$, consists of all those random
trades (including the burning of assets) which are available at time
$t$. Thus we can think of $\K_t$ as consisting of all those random
vectors which live (almost surely) in a random closed convex cone
$\K_t(\omega)$, where, denoting the $i$th canonical basis vector of
${\mathbb R}^d$ by $e_i$, $\K_t(\omega)$ is the finitely-generated
convex (hence closed) cone with generators
$\{e_j-\pi^{i,j}_t(\omega)e_i,1\leq i\neq j\leq d;\hbox{ and
}-e_k,1\leq k\leq d\}$. We shall say that $\eta$ is a self-financing
process if $\eta_t-\eta_{t-1}\in \K_t$ for each $t$, with
$\eta_{-1}\defto 0$.

It follows that the cone of claims attainable from zero endowment is
$\K_0+\ldots+\K_T$ and we denote this by $\B(\pi)$. Note that
$-\K_t$ is the time-$t$ solvency cone of claims, i.e. all those
claims which may be traded to $0$ at time $t$. Note also that,
following Kabanov et al. \cite{kabanov3}, Schachermayer uses \lq\lq hat"
notation (which we have dropped) to stress that we are trading
physical assets and uses $-\K$ where we use $\K$.

We shall show in section \ref{ultsec} that, by adding an
extra period, we may represent $\B(\pi)\cap\Linf$ by a
coherent risk measure and a new (final) set of prices for the vector of
assets. More precisely, there exists a probability space
$(\tilde{\Om},\mathcal{\tilde{F}},\tilde{\P})$ with $\F\subset
\mathcal{\tilde{F}}$ and with $\tilde{\P}$ coinciding with $\P$ on
$\F$, a vector of strictly positive random variables
$V=(v^1,\ldots,v^d)\in\L^1(\tilde{\P};\R^d)$ and a set of
probability measures $\Q$, defined on $\tilde{\Om}$, such that:
$$
\B(\pi)\cap\Linf=\left\{X\in\Linf(\P;\R^d):\;\sup_{\bbQ\in\Q}\E_\bbQ(X.V)\leq
0\right\}.
$$


\section{Conditional coherent risk measures.}\label{cond}

In the paper we will be dealing with pricing monetary risks in the future
 and, in general, in the presence of partial information. Accordingly, we recall
in this section the definition and the main result on the
characterization of a {\bf conditional} coherent risk measure. This
concept was introduced by Wang \cite{wang} and has been further
elaborated upon within different formal approaches by Artzner et al.
\cite{art}, Riedel \cite{riedel}, Weber \cite{weber}, Engwerda et
al. \cite{eng}, Scandolo \cite{scan}, Detlefsen and Scandolo
\cite{det}.

 Let $(\Om,\mathcal F,\P)$ be a probability space with
$\mathcal F_0\subset\mathcal F$ a sub-$\sigma$-algebra. Throughout
this section we consider the mapping $\rho_0:{\Linf}(\mathcal
F)\rightarrow {\Linf}(\mathcal F_0)$.

\begin{defi}\label{d1}
(See Detlefsen and Scandolo \cite{det}) We say that the mapping
$\rho_0$ is a relevant, conditional coherent risk measure with the
Fatou property if it satisfies the following axioms:
\begin{enumerate}
\item Monotonicity: For every $X,Y\in {\Linf}(\mathcal F)$,
$$
X\leq Y\,\mbox{a.s}\;\Rightarrow \;\rho_0(X)\leq
\rho_0(Y)\;\mbox{a.s}.
$$
\item Subadditivity: For every $X,Y\in {\Linf}(\mathcal F)$,
$$
\rho_0(X+Y)\leq\rho_0(X)+\rho_0(Y)\;\mbox{a.s}.
$$
\item $\mathcal F_0$-Translation invariance: For every
$X\in {\Linf}(\mathcal F)$ and $y\in  {\Linf}(\mathcal F_0)$,
$$
\rho_0(X+y)=\rho_0(X)+y\;\mbox{a.s}.
$$
\item $\mathcal F_0$-Positive homogeneity: For every $X\in {\Linf}(\mathcal F)$
and $a\in  {\Lplus}(\mathcal F_0)$, we have
$$\rho_0(a\,X)=a\,\rho_0(X)\;\mbox{a.s}.$$
\item  The Fatou property: a.s $\rho_0(X)\leq\liminf\rho_0(X_n)$, for
any sequence $(X_n)_{n\geq 1}$ uniformly bounded by $1$ and
converging to $X$ in probability.
\item Relevance: for each set $F\in\mathcal F$ with $\P[F|\, \mathcal F_0]> 0$
a.s,
$\rho_0(1_F) > 0$ a.s.
\end{enumerate}
\end{defi}

We point out that, in accordance with our aim of interpreting $\rho_0$
as a {\bf pricing} mechanism,
 we have introduced a change of sign in Definition \ref{d1},
so $X\mapsto\rho_0(-X)$ is a conditional coherent risk measure in
the sense of \cite{det}, for example.

\begin{prop}
\label{p2}(See Detlefsen and Scandolo \cite{det}) Let the mapping
$\rho_0$ be a relevant conditional coherent risk measure satisfying
the Fatou property. Then
\begin{enumerate}
\item The acceptance set
$$
\A_0\defto\{X\in {\Linf}(\mathcal F)\;;\;\rho_0(X)\leq
0\;\mbox{a.s}\}
$$
is a weak$^*$-closed convex cone, arbitrage-free, stable under
multiplication by bounded positive $\mathcal F_0$-measurable random
variables and contains $\Linf_-(\mathcal F)$.
\item There exists a
convex set of probability measures $\Q$, all of them being
absolutely continuous with respect to $\P$ and containing at least
one equivalent probability measure, such that for every $X\in
{\Linf}(\mathcal F)$:
\begin{eqnarray}
\label{ccmana} \rho_0(X)=\mbox{ess-sup}\left\{\E_{\bbQ}(X|\,\mathcal
F_0)\;;\;\bbQ\in\Q \right\}.
\end{eqnarray}
\end{enumerate}
\end{prop}

\begin{defi}
Given a conditional coherent risk measure $\rho_0$, we define
$\Q^{\rho_0}$ as follows:
\begin{eqnarray}
\label{mana1} \Q^{\rho_0} =\left\{\bbQ\ll\P\;;\;\frac{d\bbQ}{d\P}\in
\A_0^*\right\},
\end{eqnarray}
where $\A_0^*$ is the polar cone of $\A_0$. Conversely, given $\Q$ a
collection (not necessarily closed, or convex) of probability
measures absolutely continuous with respect to $\P$, we define
\begin{eqnarray}
\label{mana}
\rho_0^\Q(X)=\mbox{ess-sup}\left\{\E_{\bbQ}(X|\,\mathcal
F_0)\;;\;\bbQ\in\Q \right\}.
\end{eqnarray}
The set $\Q^{\rho_0}$ is the largest set $\Q$ for which
$\rho_0=\rho_0^\Q$.
\end{defi}


\section{Characterization of num\'eraires}\label{numsec}
First, we do the following:
\begin{enumerate}
\item we fix a relevant, coherent risk measure with the Fatou property,
$\rho:\Linf\rightarrow \R$ with acceptance set $\A$ (recall that
$\A=\{ X:\; \rho(X)\leq 0\}$ and that $\rho(X)=\inf\{c:\;X-c\1 \in
\A\}$) and test probabilities $\Q$, a maturity time $T$ and a unit
of account $\1$ (a currency e.g pounds sterling). The unit of
account $\1$ is interpreted as a contract that pays one pound at
time $T$, i.e. a zero coupon bond with redemption value of one pound.
\item we suppose that trading is frictionless at time $T$ and then
for any claim or asset $\hatX$, we denote by $X$ its value in terms
of the unit of account $\1$ at time $T$.
\end{enumerate}

It is necessary first to characterize assets which give rise to the
same acceptance sets as $\rho$. Note that, since the proofs in this
section are almost all straightforward, we give most of them in an
appendix---any missing proofs will be found in Appendix \ref{num2}.
\subsection{The one-period model}
Recall that $\A$ is an arbitrage-free, closed,
convex cone in $\Linf$ which contains $\Linf_-$.

Since $\F_0$ is not necessarily trivial, time zero may be understood
to be some time in the future and we {\em interpret} $\A_0$ as the set
of claims acceptable at time zero, using the definition in
Proposition \ref{p2}, so that
$\rho_0(X)=\mbox{ess-sup}\left\{\E_{\bbQ}(X|\,\mathcal
F_0)\;;\;\bbQ\in\Q\right\}$, and
$\A_0=\{X\in\Linf:\,\E_\bbQ(X|\,\F_0)\leq 0\;\mbox{for
all}\;\bbQ\in\Q\}=\{X:\rho_0(X)\leq 0\hbox{ a.s.}\}$.

In this one-period market governed by the pricing mechanism
$\rho_0$, to say that $\hatv$ is a {\em num\'eraire} at time zero,
means that for any claim $\hatX$, there exists an $\F_0$-measurable
number, $\lambda$, of contracts, each paying $v$ at maturity time
$T$, such that the final position $X-\lambda\,v$ is admissible. We
think of $X-\lambda\,v$ as being obtained as the net payoff from a
futures contract which agrees to exchange $\lambda$ units of $\hat
v$ for the claim $\hat X$ at the maturity date $T$.

Thus, we {\em define} $\calN_0$, the set of all num\'eraires, at
time zero by:
\begin{equation}\label{num}
\calN_0=\{v\in\L^\infty_+:\;\Linf=\A_0+\Linf(\F_0)v\},
\end{equation}
and, since $\L^\infty_-\subset\A_0$, it follows that
\begin{equation}
\calN_0=\{v\in\L^\infty_+:\;\Linf=\A_0+\Linf_+(\F_0)v\}.
\end{equation}
Now we characterise the num\'eraires.

\begin{theorem}
\label{numthm}Suppose that $v\in\Linf_+$, then $v\in \calN_0$ if and only if
{\setlength\arraycolsep{2pt}
\begin{eqnarray*}
\lambda_0(v)&\defto &\mbox{ess-inf}_{\bbQ\in\Q}\E_\bbQ(v|\F_0)=-\rho_0(-v)>0
\,\hbox{a.s},\\
{\textbf{and}\phantom{xxxxxxxx}}&&\\
\dfrac{1}{\lambda_0(v)}&\in &\Linf.
\end{eqnarray*}
}
\end{theorem}

Given a num\'eraire, we may, of course use it as a new unit of account.
First we define a measure that prices claims (expressed in units of account
$\1$),
in terms of contracts which pay the new num\'eraire.

\begin{defi}
Let $v\in \Linf_+$. The mapping $\tau:\Linf(\F)\rightarrow
\Linf(\F_0)$ is said to be a {\em $v$-denominated, conditional,
relevant, coherent risk measure with the Fatou property, with
respect to $\F_0$} if it satisfies properties 1, 2, 4, 5 and 6 of
Definition \ref{d1} and $\F_0$-translation invariance with respect
to $v$, i.e for every $X\in \Linf(\F)$ and $y\in \Linf(\F_0)$, we
have
\begin{equation}\label{inv}
\tau(X+yv)=\tau(X)+y\;\mbox{a.s}.
\end{equation}
\end{defi}

It is easily shown that, for each $v\in \calN_0$ and each $X\in\Linf$,
the set
$$
\Theta(X,v)\defto\left\{m\in\Linf(\F_0):\;X-mv\in\A_0\right\},
$$
is closed in $\Linf$ and is a lattice with respect to the minimum
relation, i.e for all $m,m'\in\Theta(X,v)$ we have
$\min(m,m')\in\Theta(X,v)$. We may thus define the mapping
$\rho_0^v:\Linf\rightarrow\Linf(\F_0)$ by
\begin{eqnarray}
\label{**}
\rho_0^v(X)=\mbox{ess-inf}\{m\in\Linf(\F_0):\;X-mv\in\A_0\},
\end{eqnarray}
and observe that $X-\rho_0^v(X)v\in\A_0$ a.s.

\begin{example}
We fix our unit of account $\1\equiv 1\,\hbox{pound sterling}$ and
then another currency, let us say the US dollar, will be denoted by
$\hat \delta$ where $\delta$ is the dollar/pound exchange rate {\em
at maturity}. So, if we assume that $\delta\in\calN_0$, and $\hatX$
is a claim, with a value in pounds at maturity $X$, then
$\rho_0^{\delta}(X)$ is the number of contracts in dollars (the
amount in dollars promised at time 0 and payable at maturity or the
dollar-denominated futures price)  we seek as payment to make the
risk $X$ acceptable.
\end{example}
We now proceed to show that a num\'eraire has all the requisite
properties.
\begin{lem}
\label{l2}Let $v\in \calN_0$. The mapping $\rho_0^v$ defined in
(\ref{**}), is a $v$-denominated conditional, relevant, coherent
risk measure which satisfies the Fatou property. Moreover
\begin{enumerate}
\item[(i)] $ \A_0=\{X\in\Linf:\;\rho_0^v(X)\leq 0\hbox{ a.s.}\}$.
\item[(ii)] For all $X\in \Linf$,
$$
\rho_0^v(X)=\mbox{ess-
sup}\left\{\dfrac{\E_\bbQ(X|\F_0)}{\E_\bbQ(v|\F_0)}:\,\bbQ\in\Q\right\}.
$$
\end{enumerate}
\end{lem}

\begin{remark}
Let $\tau$ be a $v$-denominated conditional coherent risk measure
(with respect to $\F_0$), then $\tau=\zeta^v$ where the conditional
coherent risk measure $\zeta$ is defined by:
$$
\zeta(X)=\mbox{ess-inf}\left\{m\in\Linf(\F_0):\;\tau(X-m)\leq
0\;\hbox{a.s}\right\}.
$$
This follows from (i) of Lemma \ref{l2}.
\end{remark}

\begin{remark}
\label{r4}
Let $v\in \calN_0$ and define the set of probabilities,
$$
\Q^v\defto \left\{\R:\,\dfrac{d\R}{d\bbQ}=\dfrac{v}{\E_{\bbQ}(
v|\F_0)}\; \mbox{for some}\; \bbQ\in\Q\right\}.
$$
So for all $X\in \Linf$ we have
$$
\rho_0^v(X)=\mbox{ess-sup}_{\bbQ\in\Q^v}\E_\bbQ\left(\dfrac{X}{v}\biggl|\biggr.
\F_0\right)
\defto\rho_0^{(v)}\left(\dfrac{X}{v}\right),
$$
where the mapping $\rho_0^{(v)}$ is a conditional coherent risk
measure with test probabilities $\Q^v$. Notice that we can arrive at
a price for the claim $X$ (in terms of contracts in $v$) in two
different ways: firstly, by applying the $v$-denominated coherent
risk measure $\rho_0^v$, and secondly, by converting the value $X$
to its value (at maturity) -- $X/v$---when expressed in terms of the
new unit of account $\1'\equiv \hatv$,  and then applying
$\rho_0^{(v)}$ to the value $X/v$.
\end{remark}

\begin{remark}The value of a given risk as
provided by our coherent risk measure $\rho_0$, does not incorporate
the risk coming from a change over time in the value of the unit of
account itself; in other words, the randomness of the discount rate
is not taken into account. Now, by working with the discounted
counterpart $\trho_0$ of $\rho_0$, the value of money over time does
come into play. To express $\trho_0$ in terms of $\rho_0$ and the
interest rate $r$, we denote by $\hat{\1}_0$ the contract that
delivers $1+r$ at maturity time for one unit invested at time zero
and then $\1_0=1+r$ (expressed in units of account $\1$ at maturity
time). For a contract that pays $X$ at maturity time, we have on the
one hand, the price $\trho_0(X/(1+r))$ is the cash value payable at
time zero and on the other hand, under the assumption that $1+r\in
\calN_0$, $\rho_0^{1+r}(X)$ is the number of contracts in $\hat\1_0$
(each paying $1+r$ at maturity time) so we see that:
$$
\trho_0(X/(1+r))=\rho_0^{1+r}(X).
$$
\end{remark}

\subsection{The multi-period model}
Now, for $t=0,\, 1,\ldots, T$, we define the set of claims attainable for $0$ at
time $t$:
$$
\A_t=\{X\in\Linf\;;\;\E_\bbQ(X|\G_t)\leq 0\;\hbox{for all}\;
\bbQ\in\Q\},
$$
and $\calN_t$, the set of all num\'eraires at time $t$, is defined
as the set of $v\in\Linf_+$ such that
$\Linf=\A_t+\Linf(\G_t)\,v$. We define
$\calN=\bigcap\limits_{t=0}^{T}\calN_t$
and henceforth refer to it as the set of num\'eraires and any element of it as a
num\'eraire.
Note that $\calN_T=\{X\in \Linf_+:\, \hbox{ ess-inf }X>0\}$.

\begin{defi}
\label{defi} For all $v\in \calN$ and $t=0,\,1,\ldots,T$, we define the
mapping $\rho^v_t:\Linf(\G_T)\rightarrow\Linf(\G_t)$ by
$$
\rho^v_t(X)=\mbox{ess-inf}\left\{m\in\Linf(\G_t):\;X-mv\in\A_t\right\}.
$$
\end{defi}
Notice that $\rho^v_T(X)=\frac{X}{v}$.
\begin{cor}
\label{c1} For all $v\in \calN$ and $t=0,\,1,\ldots,T$, the mapping
$\rho^v_t$ as defined in Definition \ref{defi}, is a $v$-denominated,
conditional, relevant,
coherent risk measure with the Fatou property (with respect to $\G_t$), given by
$$
\rho^v_t(X)=\mbox{ess-sup}\left\{\dfrac{\E_\bbQ(X|\,\G_t)}{\E_\bbQ(v|\,\G_t)};\,\bbQ\in\Q\right\},
$$
and
$$
\A_t=\{X\in\Linf:\;\rho^v_t(X)\leq 0 \;\hbox{a.s}\}.
$$
\end{cor}
\begin{proof}Immediate consequence of Lemma \ref{l2}.
\end{proof}

\begin{lem}
\label{l7}For all $X\in\Linf$, $0\leq t\leq t+s\leq T$ and $v,w^1,...,w^n\in
\calN$ we have:
$$\rho^v_t(X)\leq
\rho^v_t\left(\sum_{i=1}^n \rho^{w^i}_{t+s}(X_i)w^i\right),
$$
whenever $X_1,\ldots,X_n\in \Linf$ and $X=X_1+...+X_n$.
\end{lem}
\begin{proof}Similar to the proof of assertion $2$ in Lemma
\ref{l3}.
\end{proof}
\begin{remark}\label{r7}
In particular, if $v\in\calN$, then
$$\rho^v_t(X)\leq \rho^v_t\left(\rho^v_{t+1}(X)v\right).
$$
\end{remark}

\section{Time-consistency properties}\label{consist}
As we discussed in the introduction, the essential element of
pricing or hedging
in a financial market is to build a financing strategy that starts
with the price of a claim and ends with a value equal to the claim itself
at maturity. Speaking loosely, if this strategy is built by trading in a
specific set of assets $V=(v^1,\ldots,v^d)$, we shall say that the
claim is {\em represented} by the vector $V$.

Delbaen \cite{delbaen2} gave the following
\begin{defi}
A coherent risk measure $\rho$ is said to be
{\em time-consistent} if for all $0\leq t\leq t+s\leq T$ we have
$\rho_t\circ\rho_{t+s}=\rho_t$.
\end{defi}
In \cite{JB}, Jacka and Berkaoui proved that within a \lq\lq
coherent risk measure market", the property of time-consistency of
$\rho$ is equivalent to
saying that any bounded claim is represented by the unit of account
$\1$. More precisely, for any claim $X\in\Linf$, there is a sequence $(X_t)_{t=0,\,1,\ldots,T-1}$, with $X_t\in
\A_t\cap\Linf(\G_{t+1})$ for each $t$,  such that
$X=\rho(X)+X_0+\ldots+X_{T-1}$. We can think of $X_t$ as being the net
payoff at time $t+1$ from a contract entered into at time $t$ or, in the
context of an insurance company making reserves or a financial
institution marking to market, $X_t$ is the difference between reserves
for the claim $X$
at times $t$ and $t+1$.

In this section we generalize this concept and define strong
and weak time-consistency with respect to a portfolio of
num\'eraires $U\subset\calN$.

\begin{defi}
\label{defi2}{\bf Weak time-consistency} Let $U\subset \calN$. We
say that $\A$ is weakly $U$-time-consistent if for each $v\in U$,
$t\in\{0,\,1,\ldots,T-1\}$ and $X\in \Linf$, there exist sequences $X_n\in
\Linf$, $Y^n_{t+1}\in\Linf(\G_{t+1};\R^n)$ and
$V^n=(v^1,\ldots,v^n)$ (where $\{v^1,\ldots,v^n\}\subset U$) such that
\begin{itemize}
\item[(i)]the sequence $X_n$
converges weakly$^*$ to $X$ in $\Linf$,
\item[(ii)]$X_n-Y^n_{t+1}.V^n\in\A_{t+1}$,
\end{itemize}
and
\begin{itemize}
\item[(iii)]$\rho^v_t(X)=\liminf_{n\rightarrow
+\infty}\rho^v_t\left(Y^n_{t+1}.V^n\right)$.
\end{itemize}
In particular we say that $\A$ is weakly $v$-time-consistent when
$U=\{v\}$ and weakly time-consistent when $U=\{1\}$.
\end{defi}

Notice that if properties (i)--(iii) in Definition \ref{defi2} hold for some
$v\in U$ and for all $t$ then they hold for each $v\in U$.

We shall show that this definition generalises Delbaen's in Theorem
\ref{cons2}.

\begin{example}
The coherent risk measure $\rho$, associated to a singleton test
probability set $\Q=\{\P\}$, is weakly time-consistent.
\end{example}

\begin{example}\label{egcon}
Consider a binary branching tree with two branches. So
$\Omega=\{1,2,3,4\}$, $\F_0$ is trivial, $\F_1=\sigma(\{1,2\},\{3,4\})$
and $\F_2=2^\Omega$. Equating each probability measure $\bbQ$ on $\Omega$ with
the vector of probability masses
$(\bbQ(\{1\}),\bbQ(\{2\}),\bbQ(\{3\}),\bbQ(\{4\}))$, take
$\Q=co(\bQ_1,\bQ_2)$, where
$\bQ_1=(\frac{1}{3},\frac{1}{6},\frac{1}{4},\frac{1}{4})$ and $\bQ_2=
(\frac{1}{2},\frac{1}{8},\frac{3}{16},\frac{3}{16})$ (here $co$ denotes
the convex hull).
Let $\rho$ be the associated coherent risk measure.

Denoting $X\in L^\infty$ by the corresponding (lower case) vector
$(x_1,x_2,x_3,x_4)$ (so that $X(i)=x_i$) we see that
$$
\rho(X)=\max\left(\frac{1}{3} x_1+\frac{1}{6} x_2+ \frac{1}{4} x_3+
\frac{1}{4} x_4,
\frac{1}{2}x_1+\frac{1}{8}x_2+\frac{3}{16}x_3+\frac{3}{16}x_4\right),
$$
and
$$\rho_1(X)(\omega)=\begin{cases}
\max (\frac{2}{3}x_1
+\frac{1}{3}x_2,\frac{4}{5}x_1+\frac{1}{5}x_2):&\omega\in\{1,2\}\\
\half x_3+\half x_4:&\omega\in\{3,4\}\\
\end{cases}
$$

Take $X=(3,4,0,0)$ to see that $\rho\circ \rho_1\neq \rho$:
$$\rho_1(X)=
\begin{cases}
\max(\frac{10}{3},\frac{16}{5})=\frac{10}{3}:&\omega\in\{1,2\}\\
0:&\omega\in\{3,4\},\\
\end{cases}
$$
and $\rho(\rho_1(X))=\max(\frac{5}{3},\frac{25}{12})=\frac{25}{12}$,
whereas $\rho(X)=\frac{5}{3}$.

Now, setting $v=1+1_{\{1\}}$ and $\tilde x=\half (x_3+x_4)$, it is easy to check
that
$$
X=W_X+Z_X+\Delta_X,
$$
where
$$
W_X=2x_2-x_1,\; \Delta_X=\half (x_3-x_4)(1_{\{3\}}-1_{\{4\}}),\; Z_X=((x_1-
x_2)1_{\{1,2\}}+(x_1+\tilde x-2x_2)1_{\{3,4\}})v.
$$
We claim that $\rho$ is weakly $(1,v)$-time-consistent.

\begin{proof}
To check this, first take $V^n=V=(1,v)$ and $Y^n_2=Y_2=(X,0)$ for
each $n$, then $X- Y_2.V=0\in \A_2$. Now take
$Y^n_1=Y_1=(W_X,\frac{Z_X}{v})$ for each $n$, then
$X-Y_1.V=\Delta_X$ and $\rho_1(\Delta_X)=0$, so $\Delta_X\in \A_1$.
Finally, $\rho_1(X)=\rho_1(Y_2.V)$ (obviously) while it is easy to
see that $\rho(Y_1.V)=\rho((x_1,x_2,\tilde x,\tilde x))=\rho(X)$ so
the result follows.
\end{proof}
\end{example}

\begin{defi}
\label{defi1}{\bf Strong time-consistency} Let $U\subset \calN$. We
say that $\A$ is strongly $U$-time-consistent (or strongly
time-consistent with respect to $U)$ if for all $v\in U$,
$t=0,\,1,\ldots,T-1$ and $X\in \Linf$, there exist sequences $X_n\in
\Linf$, $Y^n_{t}\in\Linf(\G_{t};\R^n)$ and
$V^n=(v^1,\ldots,v^n)\subset U$ such that \begin{itemize}
\item[(i)]the sequence $X_n$
converges weakly$^*$ to $X$ in $\Linf$,
\item[(ii)]$X_n-Y^n_{t}.V^n\in\A_{t+1}$,
\end{itemize}
and
\begin{itemize}
\item[(iii)]$\rho^v_t(X)=\liminf_{n\rightarrow
+\infty}\rho^v_t\left(Y^n_{t}.V^n\right)$.
\end{itemize}
We say that $\A$ is strongly $v$-time-consistent when
$U=\{v\}$ and strongly time-consistent when $U=\{1\}$.
\end{defi}
\goodbreak
\begin{remark}
Strong time consistency says that we may trade at each time $t$ in assets in
$U$ to approximate $X$ for an initial endowment which approximates
the `price' of $X$.
\end{remark}
\begin{example}
Define $\Omega=\{0,1\}$, $T=1$, $\P$ uniform and $v(0)=1,\;v(1)=2$. Then the
coherent risk measure $\rho$, associated to the singleton test probability
set
$\Q=\{\P\}$, is strongly $(1,v)$-time-consistent.
\end{example}

\begin{proof}For any $X\in\Linf$, there exists $\al,\be\in\R$ such that $
X=\al+\be\,v$. Then $\rho_0(X)=\rho_0(\al+\be\,v)$ and
$X-(\al+\be\,v)=0\in\A_{1}$.
\end{proof}

\begin{example}
Suppose that $\P$ is the unique EMM for a (vector) price process $S\in\L^\infty$, then the
coherent risk measure $\rho$, associated to the singleton test probability
set
$\Q=\{\P\}$, is strongly $(S_T)$-time-consistent.
\end{example}
\begin{proof}
Notice that $\A_t=\{Z\in L^\infty:Z_t\defto \E_{\P}[Z|\F_t]\leq 0\}$,
and $\rho_t^v(Z)=\frac{\E_{\P}[Z|\F_t]}{\E_{\P}[v|\F_t]}$.
It follows from martingale representation that for each $X\in L^\infty$ there is
an adapted, self-financing, process
$(\theta_t)_{t=0,\ldots T}$ such that
$$
X_t\defto \E_{\P}[X|\F_t]=\theta_0.S_0+\sum_{s=0}^{t-
1}\theta_s.(S_{s+1}-S_s).
$$
Moreover, since $\theta$ is self-financing, it follows that
$$X_t=\theta_0.S_0+\sum_{s=0}^{t-
1}\theta_s.(S_{s+1}-S_s)+\sum_{s=0}^{t-
1}(\theta_{s+1}-\theta_s).S_{s+1}=\theta_t.S_t.
$$
Consequently, if we set $Y^n_t=\theta_t$ and $V^n=S_T$, then
$$
\E_{\P}[X-Y_t.V|\F_t]=\E_{\P}[X-\theta_t.S_T|\F_{t+1}]
=X_{t+1}-\theta_t.S_{t+1}=(\theta_{t+1}-\theta_t).S_{t+1}=0\in\A_{t+1},
$$
since $\theta$ is self-financing. Moreover,
$$
\rho^v_t(X)=\frac{\E_{\P}[X|\F_t]}{\E_{\P}[v|\F_t]}=\frac{\theta_t.S_t}{v_t}=
\frac{\E_{\P}[\theta_t.S_T|\F_t]}{v_t}=\frac{\E_{\P}[Y_t.S_T|\F_t]}{v_t}=\rho^v_t(Y_t.V),
$$
establishing that $\Q$ is strongly $(S_T)$-time-consistent.
\end{proof}

Remark that Definitions  \ref{defi2} and \ref{defi1} can be unified
in a single definition.
\begin{defi}
\label{defi3} Let $U\subset \calN$ and $\eta=0$ or $1$. We say that $\A$
is $(\eta,U)$-time-consistent if for all $v\in U$, $t=0,\,1,\ldots,T-\eta$ and
$X\in \Linf$, there exists $X_n\in \Linf$,
$Y^n_{t+\eta}\in\Linf(\G_{t+\eta};\R^n)$ and
$V^n=(v^1,\ldots,v^n)\subset U$ such that $X_n$ converges weakly$^*$
to $X$ in $\Linf$,
$$
X_n-Y^n_{t+\eta}.V^n\in\A_{t+1},
$$
and
$$
\rho^v_t(X)=\liminf_{n\rightarrow
+\infty}\rho^v_t\left(Y^n_{t+\eta}.V^n\right).
$$
In particular we say that $\A$ is $(\eta,v)$-time-consistent when
$U=\{v\}$ and $\eta$-time-consistent when $U=\{1\}$. The weak and
strong versions of time-consistency are obtained respectively by
setting $\eta=1$ and $\eta=0$.
\end{defi}

From the previous definitions, it's clear that strong
time-consistency implies weak time-consistency. Weak
time-consistency can be restated in the following way.

\begin{theorem}\label{cons2}
Let $U\subset \calN$. The following statements are equivalent:
\begin{itemize}
\item[(i)] the cone $\A$ is weakly $U$-time-consistent;
\item[(ii)] for all $v\in U$, $t=0,\,1,\ldots,T-1$ and $X\in \Linf$, there
exist sequences $(X_{n,1},\ldots,X_{n,n})$, with each $X_{n,i}\in \Linf$, and
$V^n=(v^1,\ldots,v^n)$, with each $v^i\in U$, such that the sequence $X_n\defto
X_{n,1}+\ldots+X_{n,n}$ converges weakly$^*$ to $X$ in $\Linf$ and
$$
\rho^v_t(X)=\liminf_{n\rightarrow +\infty}\rho^v_t\left(\sum_{i=1}^n
\rho^{v^i}_{t+1}(X_{n,i})v^i\right).
$$
\end{itemize}
\end{theorem}
\begin{proof}The implication (ii) $\Rightarrow$ (i) follows when we take
$$
Y^n_{t+1}=\left(\rho^{v^1}_{t+1}(X_{n,1}),\ldots,\rho^{v^n}_{t+1}(X_{n,n})\right).
$$
For the  implication (i)$\Rightarrow$ (ii), first consider $X\in \Linf$
with $\rho^v_t(X)=0$.
By assumption, there exist sequences $X_n\in \Linf$,
$Y^n_{t+1}\in\Linf(\G_{t+1};\R^n)$ and $V^n=(v^1,\ldots,v^n)\subset
U$ such that the sequence $X_n$ converges weakly$^*$ to $X$ in
$\Linf$, and, defining
$$
Z^n\defto X_n-Y^n_{t+1}.V^n,
$$
we have
$$Z^n\in\A_{t+1}
$$
and
\begin{equation}\label{rho}
0=\rho^v_t(X)=\liminf\rho^v_t\left(Y^n_{t+1}.V^n\right).
\end{equation}
By the Fatou property we have:
$$
\rho^v_t(X)\leq\liminf\rho^v_t(X_n)=\liminf\rho^v_t\left(Y^n_{t+1}.V^n+Z^n\right).
$$
Expressing $Y^n_{t+1}.V^n+Z^n$ as
$(Y^n_{t+1}.V^n+\rho^v_{t+1}(Z^n)v)+(Z^n-\rho^v_{t+1}(Z^n)v)$
we obtain, by subadditivity,
$$
0=\rho^v_t(X)\leq
\liminf\left\{\rho^v_t\left(Y^n_{t+1}.V^n+\rho^v_{t+1}(Z^n)v\right)+
\rho^v_t\left( Z^n-\rho^v_{t+1}(Z^n)v\right)\right\}.
$$
Now, since $Z^n\in \A_{t+1}$, it follows that
$\rho^v_{t+1}(Z^n)\leq 0$.
Then, by Remark \ref{r7}
$$
\rho^v_t\bigl( Z^n-\rho^v_{t+1}(Z^n)v\bigr)\leq \rho^v_t\bigl(v
\rho^v_{t+1}\left( Z^n-\rho^v_{t+1}(Z^n)v\right)\bigr)=0.
$$
It follows that
$$
0=\rho^v_t(X)\leq
\liminf\rho^v_t\left(Y^n_{t+1}.V^n+\rho^v_{t+1}(Z^n)v\right)\leq
\liminf\rho^v_t\left(Y^n_{t+1}.V^n\right)
=0,
$$
the second inequality following from subadditivity and the fact that
$Z^n\in \A_{t+1}$, while the last equality is directly from equation
\ref{rho}.

Now, defining $X_{n,i}=Y^{n,i}_{t+1}v^i$ for $i=1,\ldots,n$ and
$X_{n,0}=Z^n$ with $v^0=v$, we see that
$$
0=\rho^v_t(X)=
\liminf\rho^v_t\left(\sum_{i=0}^n\,\rho^{v^i}_{t+1}(X_{n,i}).v^i\right),
$$
with
$$
\sum_{i=0}^n\,X_{n,i}=Y^{n}_{t+1}.V^n+Z^n=X_n,
$$
so we have established the implication in the case where
$\rho^v_t(X)=0$.

Now for a general $X\in\Linf$, define $\tX=X-\rho^v_t(X)v$. From
the above we have
$$
\rho^v_t(\tX)=
\liminf\rho^v_t\left(\sum_{i=0}^n\,\rho^{v^i}_{t+1}(\tX^{n,i}).v^i\right),
$$
with the sequence $\tX^n\defto \sum_{i=0}^n\,\tX^{n,i}$ converging
weakly$^*$ to $\tX$ in $\Linf$. We deduce, using $\F_0$-translation invariance
with respect to
$v$ (equation (\ref{inv})), that
$$
\rho^v_t(X)=
\liminf\rho^v_t\left(\sum_{i=0}^n\,\rho^{v^i}_{t+1}(\tX^{n,i}).v^i+\rho^v_{t+1}(\rho_t^v(X)v)v\right),
$$
with the sequence $\sum_{i=0}^n\,\tX_{n,i}+\rho_t^v(X)v$ converging
weakly$^*$ to $X$ in $\Linf$.
\end{proof}

\begin{example}
Recall Example \ref{egcon}. It is straightforward to check that
$v\rho^v_1(Z_X)=Z_X$ while $\rho_1(W_X+\Delta_X)=W_X$ and
$\rho(W_X+Z_X)=\rho(X)$ so that (as we saw before) $\rho$ is weakly
$(1,v)$-time-consistent.

\end{example}

\begin{remark}
We shall prove later in Theorem \ref{count0} that the weak
time-consistency introduced in Definition \ref{defi2} is equivalent
to that introduced by Delbaen in \cite{delbaen2}.
\end{remark}
\begin{example}
Consider a binary branching tree with two branches. So
$\omega=\{1,2,3,4\}$, $\F_0$ is trivial, $\F_1=\sigma(\{1,2\},\{3,4\})$
and $\F_2=2^\Omega$, and take $\P$ uniform.
Equating each probability measure $\bbQ$ on $\Omega$ with
the corresponding vector of probability masses, define the set
$$
\Q=co\left(\left\{\dfrac{1}{4},\dfrac{1}{4},\dfrac{1}{4},\dfrac{1}{4}\right\},
\left\{\dfrac{1}{4},\dfrac{1}{4},\dfrac{3}{8},\dfrac{1}{8}\right\},
\left\{\dfrac{3}{8},\dfrac{1}{8},\dfrac{1}{4},\dfrac{1}{4}\right\},\left\{\dfrac{3}{8},\dfrac{1}{8},
\dfrac{3}{8}, \dfrac{1}{8}\right\}\right),
$$
Then the associated coherent risk measure $\rho$, is weakly
time-consistent.
\end{example}
\begin{proof}
For $\F_1$-measurable $Y$, we have
$$
\rho(Y)=\half (Y(1)+Y(3))=\half (Y(2)+Y(4)),
$$
while for all $X$
$$
\rho_1(X)(\omega)=
\begin{cases}\max\left(\frac{3X(1)+X(2)}{4},\frac{X(1)+X(2)}{2}\right):&\hbox{for
}\omega\in\{1,2\}\cr
\max\left(\frac{3X(3)+X(4)}{4},\frac{X(3)+X(4)}{2}\right):&\hbox{for
}\omega\in\{3,4\}.\cr
\end{cases}
$$
It easily follows that
$\rho=\rho\circ\rho_1$.
\end{proof}

In \cite{delbaen2}, it was shown that weak time-consistency (with
respect to $1$) is equivalent to m-stability of the corresponding
test probabilities. In order to generalise this result to our context,
we define m-stability with respect to a portfolio of assets in a similar way.

\begin{defi}{\bf Weak m-stability} Let $U\subset \calN$ and $P\subset \L^1_+$.
We say that
$P$ is weakly $U$-m-stable if for all $t=0,\,1,\ldots,T$, whenever
$Z^1,\ldots,Z^k\in P$ are such that there exists $Z\in P$, a
partition $F^1_t,\ldots,F^k_t\in\F_t$,
$\al^1,\ldots,\al^k\in\L^0(\F_t)$ with each $\al^i\,Z^i\in \L^1$ and
$Y\defto \sum_{i=1}^{k}1_{F^i_t} \al^i\,Z^i$ satisfies $\E(Z
v|\,\G_{t})=\E(Y v|\G_t)$, for all $v\in U$, then we have $Y\in P$.
In particular we say that $P$ is weakly $v$-m-stable when $U=\{v\}$
and weakly m-stable when $U=\{1\}$.
\end{defi}

\begin{example}
\label{ex4} Let $\bM(S)$ denote the set of all EMMs of a
strictly positive bounded $\R^d$-valued process $(S_t)_{t=0,\,1,\ldots,T}$
with $S^1_.\equiv 1$. Then $\bM(S)$ is weakly m-stable, on identifying
probability measures with their densities with respect to $\P$.
\end{example}

\begin{proof}
Let $\bbQ,\bbQ^1,\ldots,\bbQ^k$ be in $\bM(S)$. Let their respective
densities be $Z,Z^1,\ldots,Z^k$, with each $Z^i>0$ and fix the partition
$F^1_t,\ldots,F^k_t\in\F_t$. Define the probability measure $\R$
having the following density:
$$
Y=\sum_{i=1}^k 1_{F^i_t} \dfrac{Z_t}{Z^i_t}\;Z^i.
$$
Then for $s\geq t$ we have $\E_\R(S_{s+1}|\G_s)=\sum_{i=1}^k
1_{F^i_t} \E_{\bbQ^i}(S_{s+1}|\G_s)=S_s$ and for $s<t$ we have
$$
\E_\R(S_{s+1}|\G_s)=\E(Y_tS_{s+1}|\G_s)/Y_s=\E(Z_tS_{s+1}|\G_s)/Z_s=\E_{\bbQ}(S_{s+1}|\G_s)=S_s.
$$
Thus $\R\in \bM(S)$.
\end{proof}

We may define strong m-stability in a similar fashion.
\begin{defi}{\bf Strong m-stability} Let $U\subset \calN$ and $P\subset \L^1_+$.
We say that $P$ is strongly $U$-m-stable if for all
$t=0,\,1,\ldots,T-1$, whenever $Z^1,\ldots,Z^k\in P$ are such that
there exists $Z\in P$, a partition $F^1_t,\ldots,F^k_t\in\F_t$,
$\al^1,\ldots,\al^k\in\L^0_+(\F_{t+1})$ with each $\al^i\,Z^i\in \L^1$
and $Y\defto \sum_{i=1}^{k}1_{F^i_t} \al^i\,Z^i$ satisfies
$\E(Zv|\,\G_{t})=\E(Y v|\G_t)$, for all $v\in U$, then we have $Y\in
P$. In particular we say that $P$ is strongly $v$-m-stable when
$U=\{v\}$ and strongly m-stable when $U=\{1\}$.
\end{defi}

\begin{example}
\label{ex1} Denoting by $\bM(S)$ the set of all EMMs of a
strictly positive bounded $\R^d$-valued process $(S_t)_{t=0,\,1,\ldots,T}$
with $S^1_.\equiv 1$, the set $\bM(S)$ is strongly $S_T$-m-stable, on
identifying probability measures with their densities with respect to $\P$.
\end{example}

\begin{proof}
Fix $t\in \{0,1,\ldots,T-1\}$ and let $\bbQ,\bbQ^1,\ldots,\bbQ^k$ belong
to $\bM(S)$.
Let their respective densities be $Z,Z^1,\ldots,Z^k$, with each
$Z^i>0$,
and let
$\al^i\in \L^0_+(\G_{t+1})$ with each $\al^i Z^i\in\L^1$. Fix the
partition $F^1_t,\ldots,F^k_t\in\F_t$ such that $Y\defto
\sum_{i=1}^{k}1_{F^i_t} \al^i\,Z^i$ satisfies $ \E(Z S_T|\G_t)=\E(Y
S_T|\G_t)$. Define the probability measure $\R$ by its density $Y$.
We want to prove that $\R\in \bM(S)$. Now, for $s\geq t+1$ we have:
$$
\E_\R(S_T|\G_s)=\sum_{i=1}^k 1_{F^i_t} \dfrac{\E(Z^i
S_T|\G_s)}{Z^i_s}=\sum_{i=1}^k 1_{F^i_t} \E_{\bbQ^i}(S_T|\G_s)=S_s,
$$
and for $s\leq t$ we have first
$$
\E_\R(S_T|\G_t)=\dfrac{1}{Y_t}\E(Y S_T|\G_t)=\dfrac{1}{Z_t}\E(Z
S_T|\G_t)=\E_\bbQ(S_T|\G_t)=S_t.
$$
Then
$$
\E_\R(S_T|\G_s)=\E_\R(\E_\R(
S_T|\G_t)|\G_s)=\E_\R(S_t|\G_s)=\E_\bbQ(S_t|\G_s)=S_s.
$$
\end{proof}

\begin{example}
We consider a probability space $\Om={\mathbb Z}\backslash\{0\}$
with $\P$ defined by $\P(\om)=2^{-(1+|\om|)}$, $\G=\G_2=2^\Om$,
$\G_1=\sigma\left(\{\om,-\om\};\;\om\in {\mathbb Z}_+\right)$ and
$\G_0$ trivial. Then every set of probability measures $\Q$ is
strongly $U$-m-stable, where $U=\{v^{\om};\;\om\in{\mathbb
Z}\}$ with
$v^{\om}=1+1_{\{\om\}}$.

This result follows from the fact that the only way we can have
$(Zv)_t=(Yv)_t$ for all $v\in U$ (with $Y$ non-negative and $Z\in \Q$) is if
$Y=Z$.
\end{example}

The last two definitions can also be unified in a single definition.

\begin{defi}Let $U\subset \calN$ and $P\subset \L^1_+$. We say that $P$ is
$(\eta,U)$-m-stable
with $\eta=0,1$ if for each $t=0,\,1,\ldots,T-1+\eta$, whenever
$Z^1,\ldots,Z^k\in P$ are such that there exists $Z\in P$, a
partition $F^1_t,\ldots,F^k_t\in\F_t$,
$\al^1,\ldots,\al^k\in\L^0_+(\F_{t-\eta+1})$ with each $\al^i\,Z^i\in
\L^1$ and $Y\defto \sum_{i=1}^{k}1_{F^i_t} \al^i\,Z^i$ satisfies
$\E(Z v|\,\G_{t})=\E(Y v|\G_t)$ for all $v\in U$, then we have $Y\in
P$. In particular we say that $P$ is $(\eta,v)$-m-stable when
$U=\{v\}$ and $\eta$-m-stable when $U=\{1\}$.
\end{defi}

\goodbreak
The following theorem gives some simpler conditions for m-stability

\begin{theorem}\label{equiv}
Let $U\subset \calN$ and $P\subset \L^1_+$.
The following are equivalent:
\begin{itemize}
\item[(i)]$P$ is $(\eta,U)$-m-stable;
\item[(ii)]for each $t\in\{0,\,1,\ldots,T-1+\eta\}$, whenever $Y,W\in P$ are
such that there
exists $Z\in P$, a set $F\in\F_t$,
$\al,\beta\in\L^0_+(\F_{t-\eta+1})$ with $\al Y,\beta W\in \L^1$ and
\begin{equation}\label{H}
X\defto  1_{F} \al Y+1_{F^c}\beta W
\end{equation}
satisfies $ \E(X
v|\G_{t})=\E(Z v|\G_t)$ for all $v\in U$, then we have $X\in P$.
\item[(iii)]for each $t\in\{0,\,1,\ldots,T-1+\eta\}$, whenever $Y,W\in P$ are
such that there
exists $Z\in P$, a set $F\in\F_t$,  and for each $v\in U$ there is an
$R^v_{t+1-\eta}\in L^1_+(\F_{t-\eta+1})$
 with $R^v_{t+1-\eta} Y,R^v_{t+1-\eta} W\in \L^1$, $\E[R^v_{t+1-\eta}|\F_t]=1$
and such that
\begin{equation}\label{H2}
X\defto  (Zv)_tR^v_{t+1-\eta} (1_{F} \frac{Y}{(Yv)_{t+1-\eta}}+1_{F^c} \frac{W}
{(Wv)_{t+1-\eta}})
\end{equation}
is the same for each $v\in U$, then we have $X\in P$.
\item[(iv)]For each stopping time $\tau\leq T-1+\eta$, whenever there
exist $Z$ and $W$ in $P$ and $R^v_{\tau+1-\eta}\in L^1_+(\F_{\tau+1-\eta})$
such that $\E [R^v_{\tau+1-\eta}|\F_{\tau}]=1$ and
\begin{equation}\label{H3}
X=W\frac{(Zv)_\tau R^v_{\tau+1-\eta}}{(Wv)_{\tau+1-\eta}}
\end{equation}
is the same for each $v\in U$, then we have $X\in P$.

\end{itemize}

\end{theorem}
\begin{proof}
((i)$\Leftrightarrow$ (ii)) The forward implication is trivial.
Now suppose that property (ii) holds.
Fix $t=0,\,1,\ldots,T$ and take $Z$, $Z^1,\ldots,Z^{k}$ in $P$ such
that there exists  a partition
$F^1_t,\ldots,F^{k}_t\in\F_t$,
$\al^1,\ldots,\al^{k}\in\L^0(\F_{t-\eta+1})$ with each
$\al^i\,Z^i\in \L^1$ and $Y\defto \sum_{i=1}^{k}1_{F^i_t}
\al^i\,Z^i$ satisfies $\E(Z v|\,\G_{t})=\E(Y v|\G_t)$ for all $v\in
U$. We want to prove that $Y\in P$. First using the property in
(ii) we have
$$
Y^1\defto 1_{F^1_t} \al^1\,Z^1+1_{(F^1_t)^c} Z\in P,
$$
and by induction
$$
Y^{i+1}\defto 1_{F^{i+1}_t} \al^{i+1}\,Z^{i+1}+1_{(F^{i+1}_t)^c}
Y^i\in P.
$$
We prove easily that
$$
Y^{i}=\sum_{j=1}^{i}1_{F^j_t} \al^j\,Z^j+1_{G^{i}_t} Z,
$$
with $G^i_t=\cap_{j=1}^i(F^{j}_t)^c=\cup_{j=i+1}^{k} F^j_t$ for
$i\leq k-1$ and $G_t^{k}=\emptyset$. Then $Y=Y^{k}\in P$, establishing
(i).

((ii)$\Leftrightarrow$ (iii)) Assume that (iii) holds. Observe that (\ref{H})
implies that
$$
1_F(Xv)_{t+1-\eta}=1_F\al (Yv)_{t+1-\eta}\hbox{ and
}1_{F^c}(Xv)_{t+1-\eta}=1_{F^c}\beta (Wv)_{t+1-\eta},
$$
for each $v\in U$. It follows that
$$
1_F\al =1_F\frac{(Xv)_{t+1-\eta}}{(Yv)_{t+1-\eta}}\hbox{ and }
1_{F^c}\beta=1_{F^c}\frac{(Xv)_{t+1-\eta}}{(Wv)_{t+1-\eta}}.
$$
Setting
$$
R^v_{t+1-\eta}=\frac{(Xv)_{t+1-\eta}}{(Xv)_t}=\frac{(Xv)_{t+1-
\eta}}{(Zv)_t},
$$
we see that (\ref{H2}) holds establishing (ii).

Conversely, assuming (ii), if (\ref{H2}) is satisfied, then
$$
1_F\frac{X}{Y}=1_{F}\frac{(Zv)_tR^v_{t+1-\eta} }{(Yv)_{t+1-\eta}},
$$
and
$$
1_{F^c} \frac{X}{Y}=1_{F^c}\frac{(Zv)_tR^v_{t+1-\eta}
}{(Wv)_{t+1-\eta}},
$$
for each choice of $v\in U$. Setting these
common values to $\al$ and $\beta$ respectively, we see that
(\ref{H}) holds and $(Xv)_t=(Zv)_t$ for each $v$, so that $X\in P$,
establishing (iii).

((iii)$\Leftrightarrow$ (iv)) Suppose that (iv) holds, then, setting
$$\tau=t1_{F^c}+(T-1+\eta)1_F
$$
in (\ref{H3}) we see that
\begin{equation}\label{H4}
\tilde X=1_{F^c}(Zv)_tR^v_{t+1-\eta}\frac{W}{(Wv)_{t+1-\eta}}+1_{F} (Zv)_{T-
1+\eta}R^v_T,
\end{equation}
and $\tilde X\in P$. Now take $\tilde Z=\tilde X$, $\tilde W=Y$ and
$\tilde \tau=t1_{F}+(T-1+\eta)1_{F^c}$ in (\ref{H3}).
We obtain
$$
X=(Zv)_tR^v_{t+1-\eta} (1_{F} \frac{Y}{(Yv)_{t+1-\eta}}+1_{F^c}
\frac{W}{(Wv)_{t+1-\eta}})
$$
and $X\in P$, establishing (iii).

Conversely, suppose that (iii) holds. We prove (iv) by backwards
induction on a lower bound for $\tau$. The property is immediate for
$\tau=T-1+\eta$. Now suppose that (iv) holds whenever $\tau\geq k+1$ a.s.,
and that the stopping time $\tilde \tau$ satisfies $\tilde \tau\geq k$ a.s.
Define $F^c=(\tilde\tau=k)$ (so that
$F=(\tilde\tau\geq k+1)$) and set
$$\tau^*=\tilde\tau 1_F+(T-1+\eta)1_{F^c}.
$$

Suppose that $W,Z\in P$ and $R_{\tilde\tau+1-\eta}\in \F_{\tilde\tau+1-
\eta}$ satisfy the hypotheses of (iv) then
$$
W\frac{(Zv)_{\tilde\tau} R^v_{\tilde\tau+1-
\eta}}{(Wv)_{\tau+1-\eta}}1_F,
$$
is independent of $v$  so $Y$
defined by
$$Y=W\frac{(Zv)_\tau R^v_{\tau+1-
\eta}}{(Wv)_{\tau+1-\eta}}1_F+W1_{F^c}$$
is also independent of $v$.
Now
$$Y=W \frac{(Zv)_{\tau^*}R^v_{\tau^*+1-\eta}}{(Wv)_{\tau^*+1-\eta}},$$
where $R^v_{\tau^*+1-\eta}=R^v_{\tau+1-
\eta}1_F+\frac{(Zv)}{(Zv)_{T-1+\eta}}1_{F^c}$. It is easy to check that
$\E [R^v_{\tau^*+1-\eta}|\F_{\tau^*}]=1$ so, by the inductive
hypothesis, $Y\in P$

Now substitute $Y,W,Z$ and $F$ in (\ref{H3}), with $t=k$ and $R^v_{k+1-
\eta}=\frac{(Zv)_{k+1-\eta}}{(Zv)_k}1_F+R^v_{\tilde\tau +1-\eta}1_{F^c}$
to see that $X=W \frac{(Zv)_{\tilde\tau}R^v_{\tilde\tau+1-
\eta}}{(Wv)_{\tilde\tau+1-\eta}}$ and (by (iii)) $X\in P$, which
establishes
the inductive step.
\end{proof}

\begin{example}
Recall the risk measure in Example \ref{egcon}. Any element of $\Q$ may
be written as
$\bQ_{\lambda}=(\frac{1}{3}+\frac{1}{6}\l,\frac{1}{6}-\frac{1}{24}\l,
\frac{1}{4}-\frac{1}{16}\l,\frac{1}{4}-\frac{1}{16}\l)$.

We shall show that $\Q$ is weakly $(1,v)$-m-stable by using criterion
(iv) of Theorem \ref{equiv}.

We take the uniform measure on $\Omega=\{1,2,3,4\}$ as our reference
measure $\P$. Then (equating p.m.s and their densities) $Z$ and $W$ may
be written as $\frac{d\bQ_\l}{d\P}$ and $\frac{d\bQ_\m}{d\P}$
respectively.

Now
$\frac{d\bQ_\th}{d\P}=(\frac{4}{3}+\frac{2}{3}\th,\frac{2}{3}-\frac{1}{6}\th,
1-\frac{1}{4}\th,1-\frac{1}{4}\th)$, while $\frac{d\bQ_\th}{d\P}|_{\F_1}=(1+
\frac{1}{4}\th,1+\frac{1}{4}\th,1-\frac{1}{4}\th,1-\frac{1}{4}\th )$. Since
$v=1+1_{\{1\}}$, it
follows that
$\frac{d\bQ_\th}{d\P}v=(\frac{8}{3}+\frac{4}{3}\th,\frac{2}{3}-\frac{1}{6}\th,
1-\frac{1}{4}\th,1-\frac{1}{4}\th)$; $(\frac{d\bQ_\th}{d\P}v)_1=
(\frac{5}{3}+\frac{7}{12}\th,\frac{5}{3}+\frac{7}{12}\th,
1-\frac{1}{4}\th,1-\frac{1}{4}\th)$ and $(\frac{d\bQ_\th}{d\P}v)_0=
(\frac{4}{3}+\frac{1}{6}\th,\frac{4}{3}+\frac{1}{6}\th,\frac{4}{3}+\frac{1}{6}\th,\frac{4}{3}+\frac{1}{6}\th)$.

Now suppose that $\tau$ is a stopping time and consider
$\frac{Z_\tau}{W_\tau}$ and $\frac{(Zv)_\tau}{(Wv)_\tau}$. A quick check
shows that

{\setlength\arraycolsep{2pt}
\begin{eqnarray*}\frac{Z_\tau}{W_\tau}&=&(\frac{4+2\l}{4+2\m},\frac{4-\l}{4-\m},\frac{4-
\l}{4-\m},\frac{4-\l}{4-\m})1_{(\tau=2)}\\
&&+(\frac{4+\l}{4+\m},\frac{4+\l}{4+\m},\frac{4-
\l}{4-\m},\frac{4-\l}{4-\m})1_{(\tau=1)}\\
&&+(1,1,1,1)1_{(\tau=0)},\\
\end{eqnarray*}}

while
{\setlength\arraycolsep{2pt}
\begin{eqnarray*}
\frac{(Zv)_\tau}{(Wv)_\tau}&=&(\frac{4+2\l}{4+2\m},\frac{4-\l}{4-\m},\frac{4-
\l}{4-\m},\frac{4-\l}{4-\m})1_{(\tau=2)}\\
&&+(\frac{5+\frac{7}{4}\l}{5+\frac{7}{4}\m},\frac{5+\frac{7}{4}\l}{5+\frac{7}{4}\m},
\frac{4-\l}{4-\m},\frac{4-\l}{4-\m})1_{(\tau=1)}\\
&&+(\frac{4+\half\l}{4+\half\m},\frac{4+\half\l}{4+\half\m},\frac{4+\half\l}{4+\half\m},
\frac{4+\half\l}{4+\half\m})1_{(\tau=0)}.\\
\end{eqnarray*}}
Now equating $\frac{Z_\tau}{W_\tau}$ and $\frac{(Zv)_\tau}{(Wv)_\tau}$,
we see that if $\P(\tau=0)>0$ we must have
$4+\half\l=4+\half\m\Rightarrow\l=\m$, whilst if
$\P((\tau=1)\cap(\omega\in\{1,2\}))>0$ we must
have
$\frac{5+\frac{7}{4}\l}{5+\frac{7}{4}\m}=\frac{4+\l}{4+\m}\Rightarrow
\l=\m$. Thus, either $Z=W$ or $\tau\geq 1$ and $(\tau =1)\subseteq
\{3,4\}$. Now, assuming that $Z\neq W$, since $(\tau =1)\in \F_1$ we see
that either $\tau=2$ or $\tau=1_{\{3,4\}}+2.1_{\{1,2\}}$. In either
case, $W_\tau=W$ and $Z_\tau=Z$ so that $X$, defined by
$X=W\frac{Z_\tau}{W_\tau}=W\frac{(Zv)_\tau}{(Wv)_\tau}$ is equal to $Z$
and hence $\Q$ is weakly $(1,v)$-m-stable.
\end{example}
\goodbreak
\begin{remark}
Weak m-stability (with respect to $\{1\}$) of a set
$P\subset\L^1_+$, introduced here, coincides with m-stability as
defined in Delbaen \cite{delbaen}: for all $Z,W\in P$ with $W>0$ a.s and all
stopping times $\tau$, we have
$$
X\defto Z_{\tau}\dfrac{W}{W_{\tau}}\in P.
$$
The weak-m-stability property was first established for the collection of EMMs
for
a vector valued price-process in Jacka
\cite{jacka}.
\end{remark}

\begin{remark}
It is easy to see from condition (iv) in Theorem \ref{equiv} that $P$ is
$(\eta,U)$-m-stable $\Leftrightarrow$ $Pv^{-1}$ is $(\eta,Uv)$-m-stable,
for any $v\in \calN$.
\end{remark}

\begin{lem}
\label{l6} $\A$ is strongly time-consistent with respect to $\calN$.
\end{lem}

\begin{proof}Let $X\in\Linf$ and $t=0,\,1,\ldots,T$. Then
for $\lambda=1+\|X\|_{\Linf}$ we have $v\defto X+\lambda\in\calN$
and consequently
$$
\rho_t(X)=\rho_t(Y.V),
$$
with $Y=(-\lambda,1)$ and $V=(1,v)$.
\end{proof}

\begin{remark}
In Lemma \ref{l6}, we proved that each claim $X$ can be \lq\lq
hedged" by a portfolio of two assets. Later (in Theorem \ref{separable}) we
shall prove that
hedging can be performed uniformly via a countable portfolio of assets .
\end{remark}


\section{Results on multidimensional closed convex
cones}\label{abscone}
cone $\A$ with respect to a finite portfolio
We now introduce the concept of
representation of a cone with respect to a collection of assets $V$. As we
will see later in the next section, this new concept coincides with
the concept of decomposition that we will analyze in this section.

We consider $\B$, a weak$^*$-closed convex cone in
$\Linf(\G;\R^d)$ which is arbitrage-free.

The canonical example is where $\B$ is the collection of admissible portfolios
of the assets in $V$:
$$
\A(V)=\{X\in\Linf(\G;\R^d):\, X. V\in \A\},
$$
is the set of all portfolios in (assets in the collection) $V$ that are
admissible.

In the interests of presentation, we relegate most of the proofs of results in
this section to Appendix \ref{app3}.

First we introduce some
definitions.

\begin{defi}
\label{dd4} {\bf Weak decomposition} We say that the cone $\B$ is
weakly decomposable if
$$
\B=\overline{\oplus_{t=0}^{T-1}\,\K_t(\B)},
$$
where $\K_t(\B)=\B_t\cap \Linf(\G_{t+1};\R^d)$ and
$$
\B_t=\{X\in \Linf;\;\al\,X\in \B\;\,\mbox{for all}\;\al\in
\Linf_+(\G_t)\}.
$$
We say that $(\K_t(\B))_{t=0,\ldots,T-1}$ is the weak decomposition of
the cone $\B$.
Note that we use $\oplus$ simply to denote a sum of subsets
of a vector space.
\end{defi}
\begin{remark}
Thus $\B$ is weakly decomposable if it can be obtained by one-period
bets in the various assets/currencies. See also Definition \ref{repweak} for the
motivation for
this concept.
\end{remark}

\begin{example}The acceptance set of a weakly time-consistent coherent risk
measure is weakly decomposable (see Jacka and Berkaoui \cite{JB}).
\end{example}

\begin{defi}
\label{dd1} {\bf Strong decomposition} We say that the cone $\B$ is
strongly decomposable if
$$
\B=\overline{\oplus_{t=0}^{T}\,\C_t(\B)},
$$
where $\C_t(\B)=\B_t\cap \Linf(\G_{t};\R^d)$.
We say that $(\C_t(\B))_{t=0,\ldots,T}$
is the strong decomposition of the cone $\B$.
\end{defi}
\begin{remark}
Thus $\B$ is strongly decomposable if it can be obtained by instantaneous
exchanges of assets. See also Definition \ref{repstrong} for the motivation for
this concept.
\end{remark}

\begin{example}The cone $\B(\pi)\cap\Linf$, defined in section 2, is strongly
decomposable.
\end{example}

For the sake of simplification later in the proofs we introduce a
unified definition of weak and strong decomposition.

\begin{defi}
\label{dd5} We say that the cone $\B$ is $\eta$-decomposable with
$\eta\in\{0,1\}$ if
$$
\B=\overline{\oplus_{t=0}^{T-\eta}\,\K^\eta_t(\B)},
$$
where $\K^\eta_t(\B)=\B_t\cap \Linf(\G_{t+\eta};\R^d)$. We say that
$(\K_t^\eta(\B))_{t=0,\ldots, T-\eta}$ is the $\eta$-decomposition of the cone
$\B$. Remark
that the weak and strong decomposition are respectively associated
to $\eta=1$ and $\eta=0$.
\end{defi}

Now we define $\eta$-stability in this multidimensional context:
\begin{defi}
\label{dd2}Let $D$ denote a subset in $\L^1_+(\G;\R^d)$. We say that
$D$ is $\eta$-stable with $\eta\in\{0,1\}$ if for all
$t=0,1,\ldots,T$ whenever $Z^1,\ldots,Z^k\in D$ are such that there
exists $Z\in D$, a partition $F^1_t,\ldots,F^k_t\in\F_t$ and
$\al^1,\ldots,\al^k\in\L^0(\F_{t-\eta+1})$ with each $\al^i\,Z^i\in
\L^1$ and $Y\defto \sum_{i=1}^{k}1_{F^i_t} \al^i\,Z^i$ satisfies
$\E(Z|\G_{t})=\E(Y|\G_t)$, then we have $Y\in D$.
\end{defi}
\begin{defi}\label{dt}
For all $t=0,1,\ldots,T$, we define:
$$
D_{(t)}=\overline{conv}\left\{\al\,Z:\,Z\in D,\,\al\in
\Linf_+(\G_t)\right\}.
$$
\end{defi}
\begin{theorem}\label{equiv2}
Let  $D\subset \L^1_+(\G;\R^d)$.
The following are equivalent:
\begin{itemize}
\item[(i)]$D$ is $\eta$-stable;
\item[(ii)]for each $t\in\{0,\,1,\ldots,T-1+\eta\}$, whenever $Y,W\in D$ are
such that there
exists $Z\in D$, a set $F\in\F_t$,
$\al,\beta\in\L^0_+(\F_{t-\eta+1})$ with $\al Y,\beta W\in \L^1$ and
\begin{equation}\label{H5}
X\defto  1_{F} \al Y+1_{F^c}\beta W
\end{equation}
satisfies $ \E(X
|\G_{t})=\E(Z |\G_t)$  then we have $X\in D$.
\item[(iii)]for each $t\in\{0,\,1,\ldots,T-1+\eta\}$, whenever $Y,W\in D$ are
such that there
exists $Z\in D$, a set $F\in\F_t$,  and for each $i\in \{1,\ldots,d\}$ there is
an
$R^i_{t+1-\eta}\in L^1_+(\F_{t-\eta+1})$
 with $R^i_{t+1-\eta} Y^i,R^i_{t+1-\eta} W^i\in \L^1$,
$\E[R^i_{t+1-\eta}|\F_t]=1$
and such that
$$
R^i_{t+1-\eta} Z^i_t\left(1_F \frac{1}{Y^i_{t+1-\eta}}+1_{F^c}
\frac{1}{W^i_{t+1-\eta}}\right)
$$
is the same for each $i$,
then $X$, given by
\begin{equation}\label{H6}
X^i=  Z^i_tR^i_{t+1-\eta} (1_{F} \frac{Y^i}{(Y^i)_{t+1-\eta}}+1_{F^c}
\frac{W^i}{(W^i)_{t+1-\eta}})
\end{equation}
is in $D$.
\item[(iv)]For each stopping time $\tau\leq T-1+\eta$, whenever there
exist $Z$ and $W$ in $D$ and $R^i_{\tau+1-\eta}\in L^1_+(\F_{\tau+1-\eta})$,
such that $\E [R^i_{\tau+1-\eta}|\F_{\tau}]=1$ and
$$
\frac{Z^i_\tau R^i_{\tau+1-\eta}}{W^i_{\tau+1-\eta}}
$$
is the same for each $i$, then $X$, defined by
\begin{equation}\label{H7}
X^i=W^i\frac{Z^i_\tau R^i_{\tau+1-\eta}}{W^i_{\tau+1-\eta}},
\end{equation}
is in $D$.

\end{itemize}

\end{theorem}
The proof is essentially the same as that of Theorem \ref{equiv}.

Now we give some results on $\eta$-stability. To facilitate this we
introduce the following equivalence relation on $\L^1(\G;\R^d)$ as
follows: $Z\equiv_{t,\eta} Z'$ if there exists $\al_t\in
\L^0_+(\G_t)$ with $\al_t\,Z'\in \L^1$ such that
$Z_{t+\eta}=\al_t\,Z'_{t+\eta}$.

\begin{lem}
\label{duality0}Let $D\subset \L^1_+$ be an $\eta$-stable cone for some
$\eta\in\{0,1\}$, then, defining
$$
M^\eta_t(D)=\{Z\in \L^1(\G;\R^d):\;Z\equiv_{t,\eta}Z'\;\mbox{for
some}\; Z'\in D\}.
$$
we have
$$D=\cap_{t=0}^{T-\eta}M^\eta_t(D).
$$

\end{lem}

\begin{proof}See Appendix \ref{app3}
\end{proof}

\begin{lem}
\label{duality1} Let $D$ be a subset in $\L^1$ and define:
$[D]=\cap_{t=0}^{T-\eta} R^\eta_t$ where
$$
R^\eta_t=\{Z\in \L^1:\;Z_{t+\eta}=Z'_{t+\eta}\;\mbox{for some}\;
Z'\in D_{(t)}\}=\overline{conv}(M^\eta_t(D),
$$
where $D_{(t)}$ is defined in Definition \ref{dt}.

Then
\begin{enumerate}
\item $[D]$ is the smallest $\eta$-stable closed convex cone in
$\L^1$, containing $D$.
\item $D=[D]$ if and only if $D$ is an $\eta$-stable closed convex cone in
$\L^1$.
\end{enumerate}
\end{lem}

\begin{proof}
See Appendix \ref{app3}
\end{proof}

Next we characterize the $\eta$-decomposability of the cone $\B$. In
what follows we shall denote the polar cone of $\B$ by $\B^*$ (note
that, by assumption, $\B$ is weak$^*$ closed and so
$\B^*\subset\L^1(\G;\R^d)$).

\begin{theorem}
\label{rafa5} $\B$ is $\eta$-decomposable $\Leftrightarrow$ $\B^*$
is $\eta$-stable.
\end{theorem}

\begin{proof}
See Appendix \ref{app3}
\end{proof}

\begin{remark} From Lemma \ref{duality0} we see that if a subset $D\subset \L^1$
is
$\eta$-stable, then its polar cone $D^*$ in $\Linf$ is
$\eta$-decomposable.
\end{remark}
We establish some further results about $\eta$-decomposability in
Appendix \ref{app2}.

Now, we give a useful characterisation of $\C_t(\B)$:
\begin{lem}
\label{lem5.3} Let $X\in \Linf(\G;\R^d)$. Then the following
assertions are equivalent:
\begin{enumerate}
\item $X\in \C_t(\B)$,
\item $ X\in \Linf(\G_t;\R^d)$ and $Z_t.X\leq 0$ a.s. for all $Z\in \B^*$.
\item $\E[W.X|\,\G_t]\leq 0$ for all $W\in \L^{1}$ such that $W_t= Z_t$ for some
$Z\in \B^*$.
\end{enumerate}
\end{lem}
\begin{proof}
See Appendix \ref{app3}
\end{proof}

\begin{defi}
\label{dd3}For a fixed $t\in\{0,\ldots,T\}$, we say that a closed convex cone
$\H$ in $\L^1(\G_t;\R^d)$ is an $\G_t$-cone (or a $t$-cone) if $\al
\H\subset \H$ for each $\al\in \Linf_+(\G_t)$.
\end{defi}
\begin{remark}
The property of being a $t$-cone is like an $\G_t$-measurable version of
the convex cone property.

It follows from Theorem 4.5 and Corollary 4.7 of \cite{JBW} that
$\H$ is a $t$-cone if and only if there is a random closed convex
cone, $C$ in $\R^d$ such that
$$
\H=\{X\in\L^1(\G_t;\R^d):\; X\in C\hbox{ a.s.}\}.
$$
\end{remark}

\begin{theorem}
\label{rafa3} $\B$ is strongly decomposable if and only if there
exists a collection $(\H_t)_{t=0}^T$, with each $\H_t$ being a
$t$-cone in $\L^1(\G_t;\R^d)$ such that:
\begin{equation}\label{eqstar}
\B^*=\bigcap_{t=0}^T\{Z\in \L^1(\G;\R^d);\,Z_t\in \H_t\}.
\end{equation}
\end{theorem}

\begin{proof}Suppose that $\B$ is strongly decomposable and define
$$
\H_t=\overline{\{\al Z_t:\,\al\in \Linf_+(\G_t),\,Z\in \B^*\}}.
$$
Thanks to Lemma \ref{lem5.3}, we have (\ref{eqstar}). Conversely, remark
that
$$
\{Z\in \L^1;\,Z_t\in \H_t\}^*\subset \C_t(\B).
$$
Indeed define $N_t=\{Z\in \L^1;\,Z_t\in \H_t\}$, and let $X\in
\Linf$ such that $\E(Z.X)\leq 0$ for all $Z\in N_t$. Since $Z-Z_t\in
N_t$ for all $Z\in \L^1$ we deduce that $X\in\Linf(\G_t)$. For all
$Z\in \B^*$ and $\al\in \Linf_+(\G_t)$, we have $\E(Z.\al X)=\E(\al
Z.X)\leq 0$ since $\al Z\in N_t$. The result follows.
\end{proof}


\section{Representation}\label{repsec}
In a frictionless market with $d$ assets $S^1,\ldots,S^d$ and under
the no-arbitrage property of the price process
$S_t=(S^1_t,\ldots,S^d_t)$, any bounded claim $X$ is represented by
these assets: i.e. there exists an $\R^d$-valued strategy $\beta_t$ and a
scalar $x$ such that:
$$
X=x+\sum_{t=0}^{T}\,\beta_t.S_T,
$$
with $\beta_t.S_t\leq 0$ a.s for all $t\in\{0,\ldots,T\}$.

In the presence of
transaction costs, we define the cone
$$
\A=\left(\B(\pi)\cap\Linf\right).V=\{X.V:\;X\in \B(\pi)\cap\Linf\},
$$
where $\B(\pi)$ and $V$ are as defined in section 2. Then any bounded
claim $X$ is represented by the contracts $v^1,\ldots,v^d$. In other words,
there
exists an $\R^d$-valued strategy $\beta_t$ and a scalar $x$ such
that:
$$
X=x+\sum_{t=0}^{T}\,\beta_t.V,
$$
with $\beta_t.Z_t\leq 0$ a.s for all $t\in\{0,\ldots,T\}$ and for all
$Z$ in the polar
of $\B(\pi)\cap\Linf$, where $Z_t\defto\E(Z|\G_t)$.

In the presence of a conditional coherent risk measure $\rho$ associated with an
acceptance set $\A$, trading  can take place between num\'eraires or portfolios
of num\'eraires. In this
section we introduce the concept of representation of the cone $\A$
with respect to a set of contracts with lifetime equal to zero or $1$: \lq weak
representation' and \lq strong representation'.

\subsection{The finite case}
We assume, for now, that the fixed portfolio
$V\subset\calN$ of $d$ assets contains the unit of account $\1$, and, indeed,
that $v_1=\1$. Recall
also that each element of $V$ is bounded and bounded away from 0 (a.s.).

Recall that $$
\A(V)=\{X\in\Linf(\G;\R^d):\, X. V\in \A\},
$$
is the set of all portfolios in (assets in) $V$ that are admissible, and
$$
\A_t(V)=\{X:\; \al X\in\A(V)\hbox{ for all }\al\in L^\infty_+(\G_t)\}.
$$

\begin{defi}\label{repweak}We say that the cone $\A$ is weakly represented by
the $\R^d$-valued vector of
assets $V$ if the cone $\A(V)$ is weakly decomposable,
i.e
$$
\A(V)=\overline{\oplus_{t=0}^{T-1}\K_t(\A,V)},
$$
where $\K_t(\A,V)\defto\A_t(V)\cap\Linf(\G_{t+1};\R^d)$.
\end{defi}

Thus weak representation means that every element of $\A$ is
attainable by a collection of one-period bets in units of $V$ at
times $0,\ldots,T-1$ and trades at times $0,\ldots,T$.

\begin{defi}\label{repstrong}
For $\eta\in\{0,1\}$, we say the cone $\A$ is strongly
represented by the $\R^d$-valued vector of assets
$V$ if the cone $\A(V)$ is strongly decomposable, i.e
$$
\A(V)=\overline{\oplus_{t=0}^{T}\C_t(\A,V)},
$$
where $\C_t(\A,V)\defto\A_t(V)\cap\Linf(\G_t;\R^d)$.

\end{defi}

Thus strong representation means that every element of $\A$ is
attainable by a collection of trades in units of $V$ at times
$0,\ldots,T$.

We again unify the two concepts in the following:

\begin{defi}\label{etadec}
We say that the cone $\A$ is $\eta$-represented  by the $\R^d$-valued vector of
assets $V$ (with $\eta\in\{0,1\}$), if the cone $\A(V)$ is
$\eta$-decomposable, i.e
$$
\A(V)=\overline{\oplus_{t=0}^{T-\eta}\K^\eta_t(\A,V)},
$$
where $\K^\eta_t(\A,V)\defto\A_t(V)\cap\Linf(\G_{t+\eta};\R^d)$.
\end{defi}

\begin{theorem}
\label{rafa2} Under our assumptions on $V$, if $D$ is a convex cone in
$L^\infty(\G)$ then, defining $D(V)=\{X\in \L^\infty(\G;\R^d):\, X.V\in
D\}$,
we have
$$
D(V)^*=D^* V\defto \{ZV:\;Z\in D^*\}.
$$
In particular,
the polar of the cone of
portfolios $\A(V)$ is given by:
\begin{equation}\label{dual}
\A(V)^*=\A^* V\defto \{ZV:\;Z\in\A^*\}.
\end{equation}
\end{theorem}
\begin{proof}
First, take $Z\in D^*$; then, for any $X\in D(V)$, $\E ZV. X\leq
0$ since $X.V\in D$. It follows that  $ZV\in D(V)^*$, and so we
conclude that
$$D(V)^*\supset VD^*.$$

To prove the reverse inclusion, denote the $i$th canonical basis vector
in $\R^d$ by $e_i$. Now note first that, since  $V.(\al (v_ie_j-v_je_i))=0$,
$\al (v_ie_j-v_je_i)\in
D(V)$ for any $\al\in L^\infty$. It follows that if $Z\in D(V)^*$ then
$Z.(v_ie_j-v_je_i)=0$ and so any $Z\in D(V)^*$ must be of the form
$WV$ for some $W\in \L^1(\G_T)$. Now given $C\in D$, take $X$ such
that $X. V=C$ (which implies that $X\in D(V)$), then $0\geq \E WV.
X=\E WC$ and, since $C$ is arbitrary, it follows that $W\in D^*$.
Hence $D(V)^*\subset VD^*.$
\end{proof}

To complete the link between the results in the previous section and
this one we state the following:

\begin{lem}
Suppose that $\B$ is a weakly$^*$ closed convex cone in $\Linf(\G;\R^d)$.

Define the cone $\B.V$ by
$$
\B.V=\{X.V:\,X\in \B\}.
$$
Then
$$
\B\subset(\B.V)(V)$$
and
$$
\B=(\B.V)(V)\hbox{ if and only if }\al(v_j\,e_i-v_i\,e_j)\in \B
\hbox{ for all }\al\in\Linf_+(\G).
$$
In this case the set $\B.V$ is weakly$^*$ closed.
\end{lem}

\begin{proof}The first inclusion is immediate from the definition of
$(\B.V)(V)$.

Define $w_{ij}\defto v_j\,e_i-v_i\,e_j$ and suppose that $\B=(\B.V)(V)$, then
for all
$\al\in\Linf_+(\G)$ we have $\al w_{ij}.V=0=0.V$. So $\al w_{ij}\in
(\B.V)(V)=\B$.

Conversely, suppose that $\al w_{ij}\in \B$ for all
$\al\in\Linf_+(\G)$ and all pairs $(i,j)$. It follows by the same
argument as in the proof of Lemma \ref{rafa2} that $\B^*=VC$, for some closed
convex cone $C\subset L^1(\G)$. Now suppose that $Z\in\B^*$, so that
$Z=VW$ for some $W\in C$. It follows that $\E Z.X=\E WV.X\leq 0$ for all
$X\in B$ and thus, since $X$ is arbitrary, that $W\in (\B.V)^*$. Hence
$\B^*\subset V(\B.V)^*$.

Now we've already observed that
$$
\B\subset(\B.V)(V)
$$
so
$$
\B^*\supset (\B.V)(V)^*.
$$
But by Lemma \ref{rafa2}, $(\B.V)(V)^*=V(\B.V)^*$ and so
$$
\B^* =V(\B.V)^*.
$$
Taking polar cones once more we see that, since $\B$ is weakly$^*$
closed,
$$
\B^{**}=\B=(V(\B.V)^*)^*=(\B.V)(V)^{**}=\overline{(\B.V)(V)}.
$$
Finally, since $\B\subset
(\B.V)(V)$, we conclude that
$$
\B=\overline{(\B.V)(V)}=(\B.V)(V).
$$
To see that $\B.V$ is closed,
let $x^n\in \B.V$ be a sequence which converges to $x$, then
$\frac{x^n}{v_1}\,e_1\in (\B.V)(V)=\B$ converges to $\frac{x}{v_1}\,e_1\in \B$
(since $\B$ is closed). Hence
$x=\frac{x}{v_1}\,e_1.V\in \B.V$.
\end{proof}

Next we give the equivalence between representation of the cone $\A$
by the finite portfolio $V$ and $V$-m-stability of its polar cone.

\begin{theorem}
\label{eqq}
$\A$ is strongly (resp. weakly) represented by $V$ if and only if
$\A^*$ is strongly (resp. weakly) $V$-m-stable.
\end{theorem}
\begin{proof}
This is an immediate consequence of Theorem \ref{rafa5} and (\ref{dual}) of
Theorem \ref{rafa2}.
\end{proof}

\begin{remark}
Example \ref{egcon} explicitly gives the weak representation of an element of
$\A$ for the given risk measure.
\end{remark}

Now we show the relationship between representation of the cone $\A$
by a finite portfolio $V$ and $V$-time consistency.

\begin{theorem}
\label{count0}Let $V\subset\calN$, then $\A$ is strongly (resp.
weakly) $V$-time-consistent if and only if it's strongly (resp.
weakly) represented by $V$.
\end{theorem}

\begin{proof}Suppose that $\A$ is
$(\eta,V)$-time-consistent. Fix $t\in\{0,\ldots,T-\eta\}$ and $X\in \A_t$, so
there exists two sequences $X_n\in \Linf$ and
$Y^n_{t+\eta}\in\Linf(\G_{t+\eta};\R^d)$ such that
$$
X_n-Y^n_{t+\eta}.V\in\A_{t+1}
$$
and the sequence $X_n$ converges
weakly$^*$ to $X$ in $\Linf$ with
$$
\rho_t(X)=\liminf\rho_t\left(Y^n_{t+\eta}.V\right).
$$
Therefore, for all $\vare>0$, there exists some $N\geq 1$ such that
for all $n\geq N$
$$
\rho_t(X)+\vare\geq \rho_t\left(Y^n_{t+\eta}.V\right).
$$
Now we can write
$X_n-\vare=(X_n-Y^n_{t+\eta}.V)+(Y^n_{t+\eta}.V-\vare)$ with
$X_n-Y^n_{t+\eta}.V\in\A_{t+1}$ and $Y^n_{t+\eta}.V-\vare\in
\K^\eta_{t}(\A,V).V$. So $X_n\in \K^\eta_{t}(\A,V).V+\A_{t+1}$. By
taking the limit we see that $X\in
\overline{\K^\eta_{t}(\A,V).V+\A_{t+1}}$ and then
$\A_t=\overline{\K^\eta_{t}(\A,V).V+\A_{t+1}}$. By induction on
$t\in\{0,\ldots,T\}$ we get
$\A=\overline{\oplus_{t=0}^{T-\eta}\K^\eta_t(\A,V).V}$.

Conversely, fix $t\in\{0,\ldots,T-1\}$ and $X\in\A_t$ with $\rho_t(X)=0$, then
there exists a sequence
$$
X_n\in \oplus_{s=t}^{T-\eta}\K^\eta_s(\A,V).V,
$$
which converges weakly$^*$ to $X$ in $\Linf$. So there exists
$Y^n\in \K^\eta_t(\A,V)$ such that $Z^n=X_n-Y^n.V\in\A_{t+1}$. We
conclude that
$$
0=\rho_t(X)\leq\liminf\rho_t(X_n)=\liminf\rho_t\left(Y^n.V+Z^n\right)
\leq\liminf\rho_t\left(Y^n.V\right)\leq 0,
$$
so $\rho_t(X)=\liminf\rho_t\left(Y^n.V\right)$. Now for all $X\in
\Linf$ we have $X-\rho_t(X)\in\A_t$ and from previously
$$
\rho_t(X-\rho_t(X))=\liminf\rho_t\left(Y^n.V\right),
$$
therefore
$$
\rho_t(X)=\liminf\rho_t\left(Y^n.V+\rho_t(X)\right)=\liminf\rho_t\left((Y^n+\rho_t(X)e_1).V\right).
$$
\end{proof}

\begin{remark}
\label{count0101}The following assertions are obviously equivalent:
\begin{enumerate}
\item $\A$ is $\eta$-represented by $V$ and the cone $\K^\eta(\A,V).V$ is
closed.
\item For all $t\in\{0,\ldots,T-\eta\}$, $v\in V$ and $X\in \Linf$,
there exists some $Y\in\Linf(\G_{t+\eta};\R^d)$ such that
$X-Y.V\in\A_{t+1}$ and
$$
\rho_t^v(X)=\rho_t^v(Y.V).
$$
In particular, if either of the statements (1) or (2) holds then,
for the case of weak representation, we have that for all
$t\in\{0,\ldots,T-1\}$, $v\in
V$ and $X\in \Linf$, there exists some $X_1,\ldots,X_d\in\Linf(\G)$
such that $X=\sum_{i=1}^d\,X_i$ and
$$
\rho_t^v(X)=\rho_t^v\left(\sum_{i=1}^d\rho^{v_i}_{t+1}(X_i)v_i\right).
$$
\end{enumerate}
\end{remark}

\begin{remark}Since the cone $\K(\A,\{1\})=\K(\A)$ is closed, the
time-consistency property introduced by Delbaen is equivalent to the
weak $1$-time consistency property.
\end{remark}

Given a probability measure ${\bQ}<<\P$, denote the Radon-Nikodym
derivative (or density) of $\bQ$ with respect to $\bP$ by
$\Lambda^{\bQ}$ and denote the density of the restriction of $\bQ$
to $\F_t$ by $\Lambda^{\bQ}_t$ (so that
$\Lambda^{\bQ}_t=\E_{\bP}[\Lambda^{\bQ}|\F_t]$).

We now state the equivalence for weak representation:
\begin{theorem}\label{weakthm}
Let $V$ be a finite subset of $\calN$, then $\A$ is weakly
represented by $V$ if and only if {\bf whenever} $\bQ, \bQ'\in\Q$,
with $\bQ'\sim \P$, and $\tau$ is a stopping time such that
$$
\E_{\bQ}[V|\F_{\tau}]=\E_{\bQ'}[V|\F_{\tau}],
$$
then the p.m. $\hat\bQ$, given by
$$
\Lambda^{\hat\bQ}=\frac{\Lambda^{\bQ'}}{\Lambda^{\bQ'}_{\tau}}\Lambda^{\bQ}_{\tau}
$$
is an element of $\Q$.
\end{theorem}
\begin{proof}
This is an easy corollary of Theorems \ref{eqq} and \ref{equiv}.
\end{proof}

Now we give two key equivalences for strong representation.
\begin{theorem}
\label{strongthm} $\A$ is strongly represented by $V$ if and only if
there exists a collection $(\H_t)_{t=0}^T$, with each $\H_t$ being a
$t$-cone in $\L^1(\G_t;\R^d)$ such that:
$$
\A^*=\bigcap_{t=0}^T\{Z\in \L^1(\G);\,\E(ZV|\G_t)\in \H_t\}.
$$
\end{theorem}
\begin{proof}This is an immediate consequence of Theorem \ref{rafa3}.
\end{proof}
\begin{remark}
Thanks to Corollary 4.7 of \cite{JBW}, we can interpret a $t$-cone in
$\L^1(\G_t;\R^d)$ as the collection of all elements of $\L^1(\G_t;\R^d)$
which lie almost surely in a random, closed, convex cone in $\R^d$.
\end{remark}
\begin{example}
If $\Q=\{\P\}$ where $\P$ is the unique EMM for the vector price process
$(S_t)_{0\leq t\leq T}$ (so that $\A$ is strongly represented by $S_T$),
then, with $\H_t=\{\alpha S_t:\; \alpha\in L^\infty_+(\F_t)\}$,
$\A^*=\bigcap_{t=0}^T\{Z\in \L^1(\G);\,\E(ZS|\G_t)\in \H_t\}$. We leave
the proof of this statement to the reader.
\end{example}
\begin{example}
\label{ex2} We consider a binary branching tree again. Our sample space is
$\Om=\{1,2,3,4\}$ with
$\P$ uniform, $\G=\G_2=2^\Om$, $\G_1=\sigma(\{1,2\},\{3,4\})$ and
$\G_0$ trivial. Equating each probability measure $\bbQ$ on $\Omega$ with
the corresponding vector of probability masses, take
$$
\Q=co\left(\left\{\dfrac{1}{2},\dfrac{1}{2},0,0\right\},\left\{0,0,\dfrac{1}{2},\dfrac{1}{2}\right\},
\left\{\dfrac{1}{3},\dfrac{1}{6},\dfrac{1}{6},\dfrac{1}{3}\right\},\left\{\dfrac{1}{6},\dfrac{1}{3},\dfrac{1}{3},
\dfrac{1}{6}\right\}\right),
$$
$v=1+1_{\{1,3\}}$ and $V=(1,v)$. Then $\Q$ is strongly $V$-m-stable.
\end{example}
\begin{proof}Denoting the convex cone generated by a set $S$ by $cone(S)$,
define $D_0=cone\left(\left\{1,\frac{1}{2}\right\}\right)$,
$$
D_1(.)=cone\left(\left\{1,\frac{1}{3}\right\},\left\{1,\frac{2}{3}\right\}\right),
$$

$$
D_2(1)=D_2(3)=cone\left(\left\{1,1\right\}\right),
$$
and
$$
D_2(2)=D_2(4)=cone\left(\left\{1,0\right\}\right).
$$
It is not hard to show that
$$
\A^*=\bigcap_{t=0}^2 \left\{Z\in\L^1:\, \E(Z\,V|\G_t)\in
D_t\;\mbox{a.s}\right\}.
$$
The result follows from Theorem \ref{strongthm}.
\end{proof}

\begin{theorem}
\label{strongthm2} $\A$ is strongly represented by $V$ if and only
if whenever $\bQ, \bQ'\in\Q$, with $\bQ'\sim \P$, $\tau$ is a stopping time with
$\tau\leq T-1$ a.s.
and $\hat\bQ$ satisfies
\begin{equation}\label{str3}
\frac{\Lambda^{\hat\bQ}}{\Lambda^{\hat\bQ}_{\tau+1}}=\frac{\Lambda^{\bQ'}}{\Lambda^{\bQ'}_{\tau+1}},
\end{equation}
and
\begin{equation}\label{str4}
{\Lambda^{\hat\bQ}_{\tau}}={\Lambda^{\bQ}_{\tau}},
\end{equation}
then
$$\E_{\hat\bQ}[V|\F_\tau]=\E_{\bQ}[V|\F_\tau]
$$
implies that ${\hat\bQ}$ is an element of $\Q$.
\end{theorem}
\begin{proof}
This is an easy corollary of Theorems \ref{eqq} and \ref{equiv} on noticing that
equations (\ref{str3}) and (\ref{str4}) are equivalent to saying
that
$$
\Lambda^{\hat\bQ}=R_{\tau+1}\frac{\Lambda^{\bQ'}}{\Lambda^{\bQ'}_{\tau+1}}\Lambda^{\bQ}_\tau
$$
for some $R_{\tau+1}\in \L^1(\F_{\tau+1})$ with $\E
[R_{\tau+1}|\F_{\tau}]=1$.

\end{proof}

\begin{example}
We consider a binary branching tree on two time steps, but with one node pruned.
Thus, our sample space is $\Om=\{1,2,3\}$ with
$\P$ uniform, $\G=\G_2=2^\Om$, $\G_1=\sigma(\{1,2\},\{3\})$ and
$\G_0$ trivial. Equating each probability measure $\bbQ$ on $\Omega$ with
the corresponding vector of probability masses, take
$$
\Q=co\left((\dfrac{1}{2},\dfrac{1}{3},\dfrac{1}{6}),(\dfrac{1}{3},\dfrac{2}{9},\dfrac{4}{9})
\right),
$$
$v=1+1_{\{1\}}$ and $V=(1,v)$. Then $\Q$ is strongly $V$-m-stable.
\begin{proof}
Take a stopping time $\tau\leq 1$. Since $\F_0$ is trivial it is clear
that either $\tau=0$ a.s. or $\tau =1$ a.s.

A generic element of $\Q$
may be written as $\bQ_\l=(p_\l,q_\l,r_\l)=(\half -\frac{\l}{6},\frac{1}{3}-
\frac{\l}{9},\frac{1}{6}+\frac{5\l}{18})$. Denoting a generic p.m. on
$(\Omega,\F)$ by  $\bQ$ by
$(p, q, r)$, and taking $\bP$ to be the uniform measure
on $\Omega$ we see that
$$
\Lambda^{\bQ}=(3p,3q,3r),\;
\Lambda^{\bQ}_1=(\frac{3p+3q}{2},\frac{3p+3q}{2},3r)\hbox{ and }
\frac{\Lambda^{\bQ}}{\Lambda^{\bQ}_1}=(\frac{2p}{p+q},\frac{2q}{p+q},1).
$$
Notice that (for any value of $\l$)
$$
\frac{\Lambda^{\bQ_\l}}{\Lambda^{\bQ_\l}_1}=(\frac{6}{5},\frac{4}{5},1).
$$
First suppose that the equations (\ref{str3}) and (\ref{str4}) are satisfied
with $\bQ'=\bQ_\m$ and $\bQ=\bQ_\l$ and $\tau=1$. Then $\Lambda^{\hat
\bQ}_1=\Lambda^{\bQ_\l}_1$ and so $\hat p+\hat q=p_\l +q_\l$. Moreover,
$1+\frac{\hat p}{\hat p + \hat q}=\E_{\hat
\bQ}[v|\F_1]=\E_{\bQ_\l}[v|\F_1]=1+\frac{p_\l}{p_\l + q_\l}$. It follows
that $\hat \bQ=\bQ_\l$ and so $\hat \bQ\in \Q$.

Now suppose that $\tau=0$ and equations (\ref{str3}) and (\ref{str4}) are
satisfied
with $\bQ'=\bQ_\m$ and $\bQ=\bQ_\l$. Equating
$\frac{\Lambda^{\hat\bQ}}{\Lambda^{\hat\bQ}_1}$ and
$\frac{\Lambda^{\bQ_\m}}{\Lambda^{\bQ_\m}_1}$,
we see that $\frac{\hat q}{\hat p}=\frac{2}{3}$. Then, equating
$\E_{\hat \bQ}[v]$ and $\E_{\bQ_\l}[v]$ we see that $1+\hat p=1+p_\l$.
It follows that $\hat \bQ=\bQ_\l$ and so $\hat \bQ\in \Q$ once more.

\end{proof}
\end{example}

\subsection{The countable case}
Recall (Definition \ref{etadec}), that if $V$ is a vector of assets
in $\calN$ then
$\K^\eta_t(\A,V)\defto\A_t(V)\cap\Linf(\G_{t+\eta};\R^d)$ and
 $\A$ is $\eta$-represented by $V$ iff
$\A(V)=\overline{\oplus_{t=0}^{T-\eta}\K^\eta_t(\A,V)}$. It is clear that
this is true if and only if $\A=[\A,V]^\eta\defto
V.\overline{\oplus_{t=0}^{T-\eta}\K^\eta_t(\A,V)}$

We extend this as follows:

\begin{defi}
Let $U\subset\calN$. We say that $\A$ is $\eta$-represented by $U$
if for all $X\in\A$, there exists a sequence $X_n\in \A$ which
converges weakly$^*$ to $X$ in $\Linf$ such that for all $\vare>0$,
there exists $n\geq 1$ and a finite set $V^\vare\subset U$ such that
$X_n-\vare\in [\A,V^\vare]^{\eta}$.
\end{defi}

Next we prove the equivalence between the $U$-time-consistency of
the cone $\A$ and $U$-stability of its polar cone $\A^*$ when $U$ is
countable. In order to do this we introduce the following relations:

\begin{defi}For $Z,Z'\in \L^1$, we say that $Z\equiv_{t,\eta,U}Z'$ for
$\eta=0,1$ if there exists
$\al\in \L^0_+(\G_{t})$ with $\al Z'\in \L^1$ such that
$\E(Zu|\,\G_{t+\eta})=\al\,\E(Z'u|\,\G_{t+\eta})$ for all $u\in U$.
\end{defi}

\begin{prop}\label{infinite1}
Let $P\subset \L^1_+(\G)$ be an $(\eta,U)$-m-stable cone, then
$P=\cap_{t=0}^{T-\eta}M^\eta_t(P,U)$ where
$$
M^\eta_t(P,U)=\left\{Z:\;Z\equiv_{t,\eta,U}Z'\;\mbox{for some}\;
Z'\in P\right\}.
$$

\end{prop}

\begin{proof}The proof is essentially the same as that of Lemma \ref{duality0}.
\end{proof}

\begin{prop}\label{infinite2}
Let $P\subset \L^1_+(\G)$ and $U$ be a set of assets. Define
$$
[P,U]^{(\eta)}\defto\cap_{t=0}^{T-\eta}R^\eta_t(P,U),
$$
where
$$
R^\eta_t(P,U)=\left\{Z:\;Z\equiv_{t,\eta,U}Z'\;\mbox{for
some}\;Z'\in P_{(t)}\right\}.
$$
Then
\begin{enumerate}
\item $[P,U]^{(\eta)}$ is the smallest
$(\eta,U)$-m-stable closed convex cone in $\L^1$, containing $P$.
\item $P$ is an $(\eta,U)$-m-stable closed convex cone if and only if
$P=[P,U]^{(\eta)}$.
\end{enumerate}
\end{prop}
\begin{proof}
The proof follows that of Lemma \ref{duality1} very closely.
\end{proof}

From now on we fix a countable set of assets $U\subset\calN$ and a
sequence of finite sets of assets $U^n$, increasing to $U$ with
$U^0=\{1\}$.

\begin{theorem}
\label{count}Let $P$ denote a {\bf closed} convex cone in $\L^1_+$,
then the following are equivalent
\begin{itemize}
\item[(i)]$P$ is strongly (resp. weakly) $U$-stable;
\item[(ii)]there
exists a decreasing sequence $(P^n)_{n\geq 0}$ such that $P^n$ is strongly
(resp. weakly)
$U^n$-stable for each $n$ and $P=\bigcap_{n\geq 1}P^n$.
\end{itemize}
\end{theorem}

\begin{proof}Remark that the implication (ii) $\Rightarrow$ (i) is
straightforward. Now suppose that $P$ is
$(\eta,U)$-stable for $\eta\in\{0,1\}$. Define
$P^n=[P,U^n]^{(\eta)}$. We may check easily that the sequence $P^n$
is decreasing and $P\subset\bigcap_{n\geq 1} P^n$. We shall show that
$\bigcap_{n\geq 1} P^n\subset P$. Let $Z\in \bigcap_n P^n$ with $\E(Z)>0$
(if not $Z=0$ and then $Z\in P$), then for all $n\geq 1$ and for all
$t\in\{0,\ldots,T-\eta\}$, there exists $Z^{n,t}\in P_{(t)}$ such
that:
$$
\E(Zu|\,\G_{t+\eta})=\E(Z^{n,t}u|\,\G_{t+\eta})
$$
for all $u\in U^n$. For all $t$, there exists a sequence of positive
real numbers $a_t^n$ such that the sequence $f^{n,t}=\sum_{k\geq
n}a_t^{k}Z^{k,t}$ converges in $\L^1$ to some $Z^t\in P_{(t)}$. So
for all $t\in\{0,\ldots,T-1\}$, we have:
$$
\sum_{k\geq
n}a_t^{k}\E(Zu|\,\G_{t+\eta})=\E(f^{n,t}u|\,\G_{t+\eta}),
$$
for all $u\in U^n$. By taking the limit we see that the sequence
$\sum_{k\geq n}a_t^{k}$ converges to some $a_t=\E(Z^t)/\E(Z)$ and
then
$$
\E(Zu|\,\G_{t+\eta})=\dfrac{1}{a_t}\E(Z^{t}u|\,\G_{t+\eta}),
$$
for all $u\in U$. Thus $Z\in [P,U]^{(\eta)}=P$.
\end{proof}

\begin{prop}
\label{count9}$\A$ is strongly (resp. weakly) $U$-time-consistent if
and only if there exists an increasing sequence $(\A^n)_{n\geq 0}$ of acceptance
sets, with $\A_n$
strongly (resp. weakly) $U^n$-time-consistent for each $n$, such that
$\A=\overline{\bigcup_{n\geq 1}\A^n}$.
\end{prop}

\begin{proof}Suppose that $\A=\overline{\bigcup_{n\geq 1}\A^n}$ and let $\rho_t$
and $\rho_t^n$ be respectively the coherent risk measures associated
to the sets $\A_t$ and $\A^n_t$. Since for all $n\geq 1$, we have
$\A^n\subset[\A,U^n]^\eta$, we can suppose from now on that
$\A^n=[\A,U^n]^\eta$. First we prove that for all $X\in \Linf$, the
sequence $\rho^n_t(X)$ converges a.s to $\rho_t(X)$. Remark that for
all $n$, we have $X-\rho^n_t(X)\in\A^n\subset\A^{n+1}$ and then
$\rho^{n+1}_t(X-\rho^n_t(X))=\rho^{n+1}_t(X)-\rho^n_t(X)\leq 0$ and
the sequence $\rho^n_t(X)$ is decreasing. Moreover for all $n$,
$$
\rho^n_t(X)\geq \rho_t(X),
$$
Define
$$
\rho^\infty_t(X)=\liminf\rho^n_t(X)\geq \rho_t(X).
$$
Since $\A=\overline{\bigcup_{n\geq 1}\A^n}$ we deduce that
$\A_t=\overline{\bigcup_{n\geq 1}\A^n_t}$ and then for
$Y=X-\rho_t(X)\in\A_t$, there exists a sequence $Y^n\in \A_t^{k_n}$
with $k_n\geq n$ such that $Y^n$ converges weakly$^*$ to $Y$ in
$\Linf$. Therefore
$$
\rho_t^\infty(Y)\leq \liminf \rho_t^\infty(Y^n)\leq \liminf
\rho_t^{k_n}(Y^n)\leq 0.
$$
Hence $\rho_t^\infty(X)\leq \rho_t(X)$ from the translation
invariance property of $\rho_t^\infty$. We deduce that $\rho_t(X)=
\rho^\infty_t(X)$. Since each $\A^n$ is
$(\eta,U^n)$-time-consistent, for all $X\in \Linf$ and $t\in\{0,\ldots,T-1\}$,
there exists $X_{n,m}\in\Linf$ and
$Y^{n,m}\in\Linf(\G_{t+\eta},\R^n)$ such that $X_{n,m}$ converges
weakly$^*$ to $X$ in $\Linf$ when $m$ goes to infinity,
$X_{n,m}-Y^{n,m}.U^n\in\A^n_{t+1}\subset\A_{t+1}$ and
$$
\rho^n_t(X)=\liminf_m\rho^n_t\left(Y^{n,m}.U^n\right).
$$
Since $\rho^{n}_t\geq \rho_t$, we obtain
$$
\rho^n_t(X)\geq \liminf_m\rho_t\left(Y^{n,m}.U^n\right).
$$
We take the limit in $n$ and get
$$
\rho_t(X)=\liminf_{n,m}\rho_t\left(Y^{n,m}.U^n\right).
$$

Now assume that $\A$ is $(\eta,U)$-time-consistent and define
$\A^n=[\A,U^n]^{\eta}$. Fix $t\in\{0,\ldots,T-\eta\}$ and let $X\in \A_t$, so
there exists $X_{n}\in\Linf$, $Y^{n}\in\Linf(\G_{t+\eta},\R^n)$ and
an $\R^n$-valued portfolio $V^n\subset U$, containing the unit $\1$
such that $X_{n}$ converges weakly$^*$ to $X$ in $\Linf$,
$X_{n}-Y^{n}.V^n\in\A_{t+1}$ and
$$
\rho_t(X)=\liminf\rho_t\left(Y^n.V^n\right).
$$
Therefore for an arbitrary $\vare>0$, there exists some $N=N_\vare$
such that for all $n\geq N$
$$
\vare+\rho_t(X)\geq \rho_t\left(Y^{n}.V^n\right).
$$
Remark that
$$
X_n-\vare=(X_n-Y^n.V^n)+(Y^n.V^n-\vare),
$$
with $X_n-Y^n.V^n\in \A_{t+1}$ and $Y^n.V^n-\vare\in
\K^\eta_{t}(\A,V^n).V^n$. Consequently
$$X_n-\vare\in
\K^\eta_t(\A,V^n).V^n+\A_{t+1},
$$
with $V^n\subset U^n$. By backwards induction on $t$, we deduce that
for every $X\in\A$, there exists a sequence $X_n$ which converges
weakly$^*$ to $X$ and for any $\vare>0$, we have:
$$
X_n-\vare\in\overline{\bigcup_{k\geq
1}[\A,U^k]^\eta}=\overline{\bigcup_{k\geq 1}\A^k},
$$
by taking the limit in $n$ we obtain that $X\in \overline{\bigcup_{n\geq
1}\A^n}$ and then $\A=\overline{\bigcup_{n\geq 1}\A^n}$.
\end{proof}

\begin{theorem}
\label{equ}$\A$ is strongly (resp. weakly) $U$-time-consistent if
and only if $\A^*$ is strongly (resp. weakly) $U$-stable.
\end{theorem}
\begin{proof}This is an immediate consequence of Theorem
\ref{count} and Proposition \ref{count9}.
\end{proof}

\begin{theorem}
\label{equ1}$\A$ is strongly (resp. weakly) $U$-time-consistent if
and only if $\A$ is strongly (resp. weakly) represented by $U$.
\end{theorem}

\begin{proof}Suppose that $\A$ is
$(\eta,U)$-time-consistent. We show in the proof of Theorem
\ref{count9} that for all $X\in\A$, there exists a sequence
$X_n\in\A$ which converges weakly$^*$ to $X$ in $\Linf$ such that
for all $\vare>0$, there exists $n\geq 1$ and $V^n\subset U$ with
$$
X_n-\vare\in [\A,V^n]^{\eta}.
$$

Conversely let $U^n$ be a sequence of finite sets, increasing to
$U$. By assumption for all $X\in \A$, there exists a sequence
$X_n\in\A$ which converges weakly$^*$ to $X$ in $\Linf$ such that
for all $\vare>0$, there exists $n\geq 1$ and a finite subset
$V^n\subset U$ such that $X_n-\vare\in [\A,V^n]^{\eta}$. Then
$$
\A\subset\overline{\bigcup_{n\geq 1}[\A,U^n]^{\eta}},
$$
and consequently $\A$ is $(\eta,U)$-time-consistent.
\end{proof}

\begin{defi}We say that $\A$ is countably (resp. finitely) strongly (resp.
weakly) time-consistent
if there exists a countable (resp. finite) set $U\subset \calN$ such
that $\A$ is strongly (resp. weakly) $U$-time-consistent. By
analogy, we say that $\A^*$ is countably (resp. finitely) strongly
(resp. weakly) stable if there exists a countable (resp. finite) set
$U\subset \calN$ such that $\A^*$ is strongly (resp. weakly)
$U$-stable.
\end{defi}

\begin{theorem}\label{separable}
Suppose that the vector space $\L^1$ is separable, then the cone
$\A$ is countably strongly time-consistent.
\end{theorem}

\begin{proof}Thanks to Theorem \ref{equ}, we need only show that $\A^*$ is
countably strongly stable. Since we assume that the space $\L^1$ is
separable it follows that
the subset $B_+\defto\{X\in\L_+^{1}:\;X\leq 1\}$ is separable.
Denote by $H=\{u_n;\,n\geq 1\}$, a countable dense set and define
$$
U=\{1\}\cup\{1+u_{n}:\;n\geq 1\}.
$$
We may check easily that $\A^*$ is strongly $U$-stable, indeed let
$t\in\{0,\ldots,T-1\}$ and $Z,Z^1,\ldots,Z^k\in\A^*$ such that there
exists some $\al^i\in \L^0_+(\G_{t+1})$ with each $\al^i Z^i\in\L^1$
and a partition $F_t^1,\ldots,F_t^k$ satisfying $\E((Z-\sum_{i=1}^k
1_{F^i_t}\al^i Z^i)u|\G_t)=0$ for all $u\in U$, in particular
$\E(Z-\sum_{i=1}^k 1_{F^i_t}\al^i Z^i)u_n=0$ for all $n\geq 1$. We
deduce that for all $u\in B_+$, there is a sequence $u_n\in B_+$
which converges to $u$ in $\L^1$, define $f=Z-\sum_{i=1}^k
1_{F^i_t}\al^i Z^i$ and $f_N=f 1_{(|f|\leq N)}$ for an integer
$N\geq 1$, then
$$
|\E f(u_n-u)|\leq|\E f_N(u_n-u)|+|\E (f-f_N)(u_n-u)|\leq N
\E|u_n-u|+2\E |f|1_{(|f|\geq N)}.
$$
We take the limit when $n$ goes to infinity and obtain
$$
\lim_{n\rightarrow \infty}|\E f(u_n-u)|\leq 2\E|f|1_{(|f|\geq N)}.
$$
We take the limit again a $N$ goes to infinity to obtain
$$
\lim_{n\rightarrow \infty}\E f(u_n-u)=0.
$$
Then $\E(Z-\sum_{i=1}^k 1_{F^i_t}\al^i Z^i)u=0$ for all $u\in \Linf$
and thus $\sum_{i=1}^k 1_{F^i_t}\al^i Z^i=Z\in \A^*$.
\end{proof}

We conclude this subsection with the following
counterexample.

\begin{counterexample}
We take an uncountable collection of independent, identically
distributed Uniform[1,2] random variables indexed by $t\in[0,1]$, with
the usual product measure and $\sigma$-algebra. So, to be concrete:
$\Omega=[1,2]^{[0,1]}$ and we take the
coordinate process $(X_t)_{t\in [0,1]}$ with $X_t(\omega)=\omega(t)$
for $\omega\in\Omega$. We define the $\sigma$-algebra
$$
\F_1=\F=\B([1,2]^{[0,1]}),
$$
and the probability measure $\P$ is the product meaure on $\Omega$ corresponding
to Lebesgue measure
$\lambda$ on each component interval $[1,2]$. We take $\F_0$ to be the
trivial $\sigma$-agebra. Note that the vector space
$\L^1(\Omega,\F,\P)$ is {\em not} separable.

Now consider the coherent
risk measure associated with the singleton $\{\P\}$: we claim that there is no
countable set $U\subset \N$ such that $\{\P\}$ is $U$-m-stable.

To prove this, suppose that
there is such a countable set $U$, with $U=\{u_n,\,n\geq 1\}$ where each $u_n$
is $\F$-measurable.
Then for each $n$, there exists a sequence $(X_{s^n_j})_{j\geq 1}$ such
that $u_n$ is measurable with respect to $\sigma(X_{s^n_j}:\, j\geq 1)$
and so, by diagonalisation, there is a sequence $(X_{s_j})_{j\geq 1}$
such that each $u_n$ is measurable with respect to $\sigma(X_{s_j}:\,
j\geq 1)$.

Take $X=X_t$ for some $t$ with $t\notin
\{s_j:\,j\geq 1\}$. By assumption, defining $V^n=(u_1,\ldots,u_n)$, there is
\begin{itemize}
\item[(i)]a sequence $X_n\in \L^\infty$
that converges weakly$^*$ to $X$;
\item[(ii)]a sequence $Y^n\in \R^n$ (since $\F_0$ is trivial) such that
$X_n-Y^n.V^n \in \A_1$ for each integer $n$;  with
\item[(iii)]$\E_{\P}(X)=\lim_{n\rightarrow\infty}\E_{\P}(Y^n.V^n)$.
\end{itemize}
Since  $\A_1 =\L^{\infty}_-$ we have $X_n\leq Y^n.V^n$ a.s. for each
integer $n$ and so, for all $Z\in \L^1_+$ we have
$\E_{\P}(Z\,Y^n.V^n)\geq \E_{\P}(Z\,X_n)$. In particular, taking
$p\in \N$, and setting $Z_p=X^{2p}/\E(X^{2p})$,  we see that for all
$p$:
\begin{eqnarray*}
\E_{\P}(X)&=&\liminf\E_{\P}(Y^n.V^n)=\liminf\E_{\P}Z_p(Y^n.V^n)\\
&\geq&
\liminf\E_{\P}(Z_p\,X_n)=\E_{\P}(X^{2p+1})/\E_{\P}(X^{2p})\geq
\|X\|_{\L^{2p}},\\
\end{eqnarray*}
the second equality in the first line holding since $Z_p$ and
$(Y^n.V^n)$ are independent and $\E Z_p=1$ and the last inequality
in the second line is an application of H\:older's inequality
$\|X\|_{\L^{q}}\leq \|X\|_{\L^{q+1}}$ with $q=2p$. We take the limit
as $p\rightarrow \infty$ to obtain:
$$
\E_{\P}(X)=3/2\geq\hbox{ess-sup} X=2,
$$
since $X$ is uniform on $[1,2]$ under $\P$. This is the desired contradiction.
\end{counterexample}

\subsection{The case of a finite sample space.}
Here we consider the case where $\Om$ is finite with cardinality $N$. We
consider random variables as vectors in $\R^N$. It is not immediately
obvious (but is, nevertheless, true) that an acceptance set $\A$ is finitely
strongly time consistent.

\begin{lem}
The cone $\K^\eta(\A,V)$ is closed in $\Linf$.
\end{lem}

\begin{proof}
Remark that $\K^\eta(\A,V)=\K_{0,T-\eta}^\eta(\A,V)$ where
$$
\K_{t,T-\eta}^\eta(\A,V)=\K_{t}^\eta(\A,V)+\ldots+\K_{T-\eta}^\eta(\A,V),
$$
for $t\in \{0,\ldots,T-\eta\}$. We prove the closedness of the cone
$\K^\eta(\A,V)$ by backwards induction on $t=T-\eta,\ldots,0$. For
$t=T-\eta$, the cone $\K_{T-\eta}^\eta(\A,V)$ is closed. Now suppose
that the cone $\K_{s,T-\eta}^\eta(\A,V)$ is closed for
$s=T-\eta,\ldots,t+1$. Observe that
$$
\K_{t,T-\eta}^\eta(\A,V)=\K_{t}^\eta(\A,V)+\K_{t+1,T-\eta}^\eta(\A,V),
$$
and that the subset
$$
\N\defto \K_{t}^\eta(\A,V)\cap -\K_{t+1,T-\eta}^\eta(\A,V),
$$
forms a vector space. We define $\N^\perp$ to be its orthogonal complement, then
$$
\K_{t,T-\eta}^\eta(\A,V)=\K_{t}^\eta(\A,V)\cap\N^\perp+\K_{t+1,T-\eta}^\eta(\A,V).
$$
Now take a sequence
$$
x^n=x^n_0+x^n_1\in
\K_{t}^\eta(\A,V)\cap\N^\perp+\K_{t+1,T-\eta}^\eta(\A,V),
$$
that converges (weakly$^*$ or in norm) to some $x$. We claim that
the sequence $x^n_0$ is bounded. If not we divide both sides of the
equation by the norm of $x^n_0$ in $\Linf$ and obtain
$$
\dfrac{x^n}{\|x^n_0\|}=\dfrac{x^n_0}{\|x^n_0\|}+\dfrac{x^n_1}{\|x^n_0\|}\defto
y^n_0+y^n_1.
$$
The sequence $y^n_0$ is bounded, it converges (or at least some
subsequence does) to some $y_0$ and then the sequence $y^n_1$ converges
to $y_1=-y_0$. Therefore $y_0\in \N\cap\N^\perp$ which means that
$y_0=0$. This contradicts the fact that $\|y_0\|=1$. Now, since the
sequence $x^n_0$ is bounded, w.l.o.g. it converges to some $x_0$ and then the
sequence $x^n_1$ converges
to $x_1$. We conclude that the sequence $x^n$ converges to
$x=x_0+x_1$.
\end{proof}

\begin{lem}$\A$ is finitely strongly time-consistent.
\end{lem}
\begin{proof}We may assume without loss of generality that $\F$ is the
power set of $\Om$. Then every $X\in\Linf$ can be written as
$X=\sum_{\om\in\Om}X(\om)\,1_{\{\om\}}$. Define
$U=\{1,1+1_{\{\om\}};\;\om\in\Om\}$. We need to show that $\A^*$ is
strongly $U$-stable. Fix $t\in\{0,\ldots,T-1\}$,
$Z,Z^1,\ldots,Z^k\in\A^*$ such that there exists some $\al^i\in
\L^0_+(\G_{t+1})$ with each $\al^i Z^i\in\L^1$ and a partition
$F_t^1,\ldots,F_t^k$ satisfying $\E((Z-\sum_{i=1}^k 1_{F^i_t}\al^i
Z^i)u|\G_t)=0$ for all $u\in U$,

which means that $\E(Z-\sum_{i=1}^k 1_{F^i_t}\al^i
Z^i)\,1_{\{\om\}}=0$ for all $\om\in\Om$. Consequently
$Z=\sum_{i=1}^k 1_{F^i_t}\al^i Z^i$ and so $\sum_{i=1}^k
1_{F^i_t}\al^i Z^i\in \A^*$.
\end{proof}

\goodbreak
\section{Associating a coherent risk measure to a trading
cone.}\label{ultsec}
\def\B{\mathcal B}
As promised, we now show how to represent a trading cone as
(essentially) the acceptance set of a coherent risk measure\footnote{A version
of the results in this section
originally appeared in \cite{JBW}. Since they are only distantly related
to the main results in that paper, we have removed them from the version
of that paper which is to be submitted for publication.}.

Let $\B$ be a closed convex cone given by $\B=\K_0+...+\K_T$ where,
as described in section 2,
each $\K_t$ is generated by positive $\F_t$-measurable multiples of
the vectors $-e_i,\,e_j-\pi^{ij}_t\,e_i$ for $1\leq i,j\leq d$.

Recall that null strategies are elements $(\xi_0,\ldots,\xi_T)$ of
$\K_0\times\ldots \times \K_T$ staisfying $\sum_0^T\xi_t=0$, and, from
\cite{JBW}, that we may
suppose
without loss of generality that the null strategies of this decomposition form a
vector space. Our aim in this section is to transform trading with
transaction costs to a partially frictionless setting by adding a new period
on the time axis and then to show that the revised trading cone is
(essentially) the acceptance set of a coherent risk measure.

We introduce some notation. For $i\in\{1,\dots,d\}$ we define
the random variables $B^i\defto\pi^{1i}_T$ and
$S^i\defto1/\pi^{i1}_T$ We define the random convex sets
$H=\{1\}\times\prod\limits_{i=2}^d\,[S^i,B^i]$ and, for $\varepsilon>0$\,,
$$
H_\varepsilon=\{1\}\times\prod\limits_{i=2}^d\,[(1-\varepsilon)S^i,(1+\varepsilon)B^i].
$$
Let $\Psi_\varepsilon$ be the (finite) set of extreme points of the
set $H_\varepsilon$, i.e the $2^{d-1}$ random vectors of the form
$(1,X_2,...,X_d)$ where
each $X_i=(1-\varepsilon)S^i\,\mbox{or}\,(1+\varepsilon)B^i$.
Let $\tilde\Omega=\{0,1\}^{d-1}$ and enumerate the elements of
$\Psi_\varepsilon$ as follows:
$$
\Psi_\varepsilon=\{Y(\om,\tom):\,\tom\in\tOm\},
$$
where
\begin{equation}\label{ext1}
Y(\omega,\tomega_1,\ldots,\tomega_{d-1})\defto
e_1+\sum_{j=1}^{d-1}\,\left\{(1-\tom_j)(1-\vare)\,S^{j+1}(\om)
+\tom_j(1+\vare)\,B^{j+1}(\om)\right\}\,e_{j+1}.
\end{equation}

Define
$\B^o$ to be the collection of consistent price processes for $\B$.
Recall from \cite{schacher} that this means that
$$
\B^o=\{Z\in\L^1_+(\F_T,\R^d):\, Z>0\hbox{ a.s. and }Z_t\in \K_t^{(*)}\hbox{
a.s.}\},
$$
where $\K_t^{(*)}\defto \{X\in \L^1(\F_t,\R^d):\, X.Y\leq 0\hbox{ a.s.
for all  }Y\in \K_t\}$.
\begin{prop}
\label{cc} Let $Z\in \B^o$. Then there
exist strictly positive random variables $\lambda^{Z_T}(\cdot,\tom)$ defined for
each $\tom\in\tOm$
such that $\sum\limits_{\tom\in\tOm}\,\lambda^{Z_T}(\om,\tom)=1$ and
$$
\frac{Z_T}{Z^1_T}=\sum_{\tom\in\tOm}\,\lambda^{Z_T}(\om,\tom)Y(\om,\tom).
$$
\end{prop}
\begin{proof}
We know from the properties of consistent price processes that for
$i,j=1,...,d$,
$$
\frac{Z^j_T}{Z^i_T}\leq \pi^{ij}_T\leq
\pi^{i1}_T\,\pi^{1j}_T=\frac{B^j}{S^i}.
$$
In consequence, for every $i=2,\cdots,d$ we get
$$
\barZ^i_T\defto\frac{Z^i_T}{Z^1_T}\in [S^i,B^i]\,,
$$
and so
$$
\barZ_T=(\barZ^1_T,...,\barZ^d_T)\in H\subset H_\vare.
$$
Now, for $2\leq i\leq d$, let
$$
\theta(\om,i)\defto \frac{\barZ^i_T-S^i(1-\vare)}{B^i(1+\vare)-S^i(1-
\vare)},
$$
and then define
$$
\lambda^{Z_T}(\om,\tom)\defto \prod_1^{d-
1}\theta(\om,i+1)^{\tom_{i}}(1-\theta(\om,i+1))^{1-{\tom_{i}}}
$$
Since $\theta(\om,i)$ is exactly the co-efficient $\theta$ such that
$$
\barZ^i_T=\theta B^i(1+\vare)+(1-\theta)S^i(1-\vare)
$$
the result follows.
\end{proof}

To set up the new probability space, let $\tF$
be the power set of $\tOm$ and let $\tP$ be the uniform measure on
$\tOm$, then define $\hOm=\Om\times\tOm$, $\hF=\F\otimes\tF$ and
$\hP=\P\otimes\tP$.

Now we define the frictionless
bid-ask prices at time $T+1$ by
$$
\pi^{ij}_{T+1}\defto\dfrac{Y_j}{Y_i}
$$
(where the random vector $Y$ is defined in (\ref{ext1})), and so
$K_{T+1}$ is the convex cone generated by positive
$\F_{T+1}$-measurable multiples of the vectors $-e_i$ and
$e_j-\pi^{ij}_{T+1}\,e_i$, where $\F_{T+1}=\F_T\otimes \tF$. We
define the new trading cone  by $\B_{T+1}\defto\B+K_{T+1}$. Here we
assume the obvious embedding of $\B$ in $L^0(\F_{T+1};\R^d)$.

From now on, closedness and arbitrage-free properties are with
respect to the vector space $\L^0(\F_{T+1})$.

\goodbreak
\begin{prop}\label{extend}
The cone $\B_{T+1}$ is closed and arbitrage-free.
\end{prop}
\begin{proof} We prove first that a consistent price process for the cone $\B$
can be extended to (be the trace of) a consistent price process for
the cone $\B_{T+1}$.

Let $Z_T\in\B^o$ and define
$$
Z_{T+1}\defto
2^{d-1}\,Z^1_T\,\lambda^{Z_T}\,Y,
$$
where the random variable $\lambda^{Z_T}$ is given in Proposition
\ref{cc}. Then $Z_{T+1}>0$, $Z_{T+1}\in K^*_{T+1}$ and for $X_T\in
\L_+^{\infty}(\F_T)$ we have, by Fubini's Theorem,
$$
\E_{\hP}\,(X_T.Z_{T+1})
=\sum_{\tom\in\tOm}\,\E_{\P}\,X_T.\left(Z^1_T\,\lambda^{Z_T}(\cdot,\tom)
Y(\cdot,\tom)\right)=\E_{\P}\,X_T.\left(Z^1_T\bar
Z_T\right)=\E_{\P}\,(X_T.Z_{T}).
$$
Consequently $Z_{T+1}\in \L^1$ with $Z_T=\E_{\hP}\,(Z_{T+1}|\F_T)$,
therefore $(Z_0,\ldots,Z_T,Z_{T+1})$ is a consistent price process
for the cone $\B_{T+1}$ and so we conclude from Theorem 4.10 of
\cite{JBW} that $\bar \B_{T+1}$ is arbitrage-free. We shall now show
that
\begin{equation}\label{little}
K_{T+1}\cap \L(\F_T)\subset\B.
\end{equation}
Indeed, let $X\in K_{T+1}\cap \L(\F_T)$, so for every $n\geq 1$, we
have
$$
X^n\defto X\,1_{(|X|\leq n)}\in K_{T+1}\cap \Linf(\F_T);
$$
therefore, for any consistent price process, $Z$,
$$
\E(Z_T.X^n)=\E(Z_{T+1}.X^n)\leq 0.
$$
It follows from Theorem 4.14 of \cite{JBW} that $X^n\in\B$ and thus, by closure,
$X\in\B$.

Now we prove that the
cone $\B_{T+1}$ is closed. We do this by showing that
$\N(\trade_0\times\ldots\times \trade_{T+1})$, the collection of null strategies
of the decomposition $\trade_0+\ldots +\trade_{T+1}$, is a vector space.

Let
$$
(x_0,\ldots,x_{T+1})\in \N(\trade_0\times\ldots\times \trade_{T+1})
$$ and define $$
x=x_0+...+x_{T}
$$
so that $x+x_{T+1}=0$. Then it follows (since $x\in \L(\F_T)$) that
$x_{T+1}\in \L(\F_T)$ and so we conclude from (\ref{little}) that
$x_{T+1}\in\B$. We deduce that there exist $y_0\in K_0,\ldots,y_T\in
K_T$ such that $x_{T+1}=y_0+\ldots+y_T$. We conclude that each
$-(x_t+y_t)\in K_t$ and then, by adding $x_t$, respectively $y_t$,
we conclude that both $-x_t$ and $-y_t$ are contained in $K_t$ for
$0\leq t\leq T$.

Observe that since the time $T+1$ bid-ask prices are frictionless,
it follows that every element, $u\in \trade_{T+1}$ can be written as
$u=u_1-u_2$, where $u_1\in lin(\trade_{T+1})$, the lineality space
of $\trade_{T+1}$, and $u_2\geq 0$. If we express $x_{T+1}$ like
this, we then have that
$$
0\leq
u_2=u_1-x_{T+1}=u_1+x\in\B_{T+1},
$$
so, since $\B_{T+1}$ is arbitrage-free, $u_2=0$ and therefore
$$
-x_{T+1}=-u_1\in
\trade_{T+1}
$$ (since $u_1\in lin(\trade_{T+1}$)). It follows that
$\N((\trade_0\times\ldots\times \trade_{T+1})$ is a vector
space.
\end{proof}

Now define the subset of probabilities
$$
\Q\defto\left\{\bbQ:\,
\frac{d\bbQ}{d\hP}=2^{d-1}\,\frac{Z^1_T}{Z^1_0}\,\lambda^{Z_T}\;:\;Z\in\B^o\right\},
$$
and denote by $\rho$ the associated coherent risk measure.

\begin{theorem}\label{reverse}For every $X\in \Linf(\F_T;\R^d)$ we
have:
\begin{equation}\label{risk1}
\rho(Y.X)=\sup\{\E(Z_T.X):\;Z\in\B^o,\,\E Z^1_T=1\}.
\end{equation}

In particular
\begin{eqnarray}\label{risk2}
\lefteqn{\B\cap\Linf(\F_T;\R^d)}\\
&=&\left\{X\in\Linf(\F_T;\R^d):\,\E_\bbQ(Y.X)\leq
0\, \hbox{ for all }\;\bbQ\in \Q\right\}\nonumber\\
&=&\left\{X\in\Linf(\F_T;\R^d):\rho(Y.X)\leq 0\right\}.\nonumber
\end{eqnarray}
\end{theorem}
\begin{proof}
Equality (\ref{risk1}) is immediate from the definition of $\Q$; the second
equality in (\ref{risk2})
follows from (\ref{risk1}), while
the first follows from Theorem 4.14 of \cite{JBW} and the fact that, as in the
proof of Proposition \ref{extend},
$\E_{\bbQ}Y.X=\E_{\P}Z_T.X$
\end{proof}

\begin{remark}
If we define $\rho_t:\L^\infty(\F_{T+1})\rightarrow
\L^\infty(\F_{t})$ by
$$
\rho_t(X)=\hbox{ess inf}\{\lambda\in \L^\infty(\F_{t}):\;
\rho(c(X-\lambda))\leq 0\hbox{ for all }c\in \L^\infty_+(\F_t)\},
$$
then it is easy to show that
\begin{itemize}
\item[(i)]$\{X:\; cX\in \B_{T+1}\cap \L^\infty(\F_{T+1};\R^d)\hbox{ for all
}c\in\L^\infty_+(\F_t)\}
=\{X:\;\rho_t(Y.X)\leq 0 \hbox{ a.s.}\}$.
\item[(ii)]$\C_t^\infty\defto \C_t(\B_{T+1})\cap \L^\infty(\F_{T+1};\R^d)=\{X\in
\L^\infty(\F_{t};\R^d):\;
\rho_t(Y.X)\leq 0 \hbox{ a.s.}\}$.
\end{itemize}
It follows directly from Theorem 4.16 of \cite{JBW}, that
$\C_t^\infty$ is $\sigma(\Linf(\P),\L^1(\P))$-closed  and hence we
may apply Corollary 4.7 of \cite{JBW}.
\end{remark}


\begin{appendix}
\section{Proofs and further results on num\'eraires}\label{num2}

{\em Proof of Theorem \ref{numthm}:\ }
Suppose $v\in \calN_0$, then there exists some $\lambda\in\Linf(\F_0)$
such that $1-\lambda\,v\in\A_0$. Since, by the no-arbitrage property,
$1_F\notin\A_0$ for any
$F\in\F_0$ with $\P(F)>0$, we have $\lambda>0$ a.s.
Now, $1-\lambda\,v\in\A_0$ means that for all $\bbQ\in\Q$ we have
$1-\lambda\,\E_{\bbQ}(v|\F_0)\leq 0$, therefore a.s
$\lambda_0(v)\geq 1/\lambda>0$ and $1/\lambda_0(v)\leq \lambda$.

Now let $v\in \Linf$ be such that $\lambda\defto \lambda_0(v)>0$ a.s
and $1/\lambda\in\Linf$. Then for all $X\in \Linf$, setting
$b=\|X\|_{\Linf}$, we have
$$
X-\dfrac{b}{\lambda}\,v\in\A_0,
$$
since for all $\bbQ\in\Q$,
$$
\E_\bbQ\left(X-\dfrac{b}{\lambda}\,v|\F_0\right)\leq
\E_\bbQ\left(X-\dfrac{b}{\E_\bbQ(v|\F_0)}\,v|\F_0\right)=\E_\bbQ(X|\F_0)-b\leq
0.
$$
\hfill$\square$

{\em Proof of Lemma \ref{l2}:\ }The first assertion can be deduced
immediately from the properties of the cone $\A_0$. Now
we prove that
$$
\A_0=\{X\in\Linf:\;\rho_0^v(X)\leq 0\;\mbox{a.s}\}\defto \A^v.
$$
By definition of the mapping $\rho_0^v$ we have the first inclusion
$\A_0\subseteq \A^v$. Now let $X\in\A^v$, then
$X-\rho_0^v(X)v\in\A_0$ from the definition of $\rho_0^v$ and
$\rho_0^v(X)v\in\A_0$ from the monotonicity of $\rho_0$ and then
$X=(X-\rho_0^v(X)v)+\rho_0^v(X)v\in\A_0$.

To prove (ii), define $\xi(X)\defto \mbox{ess-
sup}\left\{\dfrac{\E_\bbQ(X|\F_0)}{\E_\bbQ(v|\F_0)}:\,\bbQ\in\Q\right\}$.
Note first that $X-\rho_0^v(X)v\in\A_0$, so for all $\bbQ\in\Q$ we
have
$$
\E_\bbQ(X-\rho_0^v(X)v|\F_0)=\E_\bbQ(X|\F_0)-\rho_0^v(X)\,\E_\bbQ(v|\F_0)\leq
0,
$$
which leads us to conclude that $\xi(X)\leq\rho_0^v(X)$. Now define
$\tX=X-\rho_0^v(X)v$ and suppose that there exists some $\vare>0$
such that $\P(F^\vare)>0$ where
$$
F^\vare\defto\{\xi(\tX)\leq -\vare\}\in\F_0.
$$
Then $(\tX+\vare v) 1_{F^\vare}\in \A_0$ and consequently
$\rho_0^v(\tX)\leq -\vare$ on $F^\vare$. This contradicts the fact
that $\rho_0^v(\tX)=0$ a.s. We conclude then that
$\xi(\tX)=\rho_0^v(\tX)=0$. By the $\F_0$-translation invariance
property with respect to $v$ of $\xi$ we conclude that
$\xi(X)=\rho_0^v(X)$.
\endpf

Later we will need the following lemma.
\begin{lem}
\label{r1} Let $u,v\in \calN_0$ and $X\in\Linf$, then
$(\rho_0^v(X)=0)=(\rho_0^u(X)=0)$ a.s.
\end{lem}

\begin{proof}Fix $X\in\Linf$, define $F=(\rho_0^v(X)=0)$ and $X^F=X1_F$. Remark
that $F\in\F_0$ and
$$
\rho_0^v(X^F)=\rho_0^v(X)1_F=0.
$$
It follows that $X^F\in \A_0$. Suppose that there exists some
$\vare>0$ such that $\P(G^\vare)>0$ with
$G^\vare=F\cap(\rho_0^u(X)\leq -\vare)$. We deduce that $(X+\vare
u)1_{G^\vare}\in\A_0$ and thus by subadditivity $0=\rho_0^v(X)\leq
-\vare \rho_0^v(u)<0$ a.s on $G^\vare\subset F$. We obtain a
contradiction. Thus, for all $\vare>0$, we have $\P(G^\vare)=0$.
Consequently, since, by part (i) of Lemma \ref{l2},
$X^F\in\A_0=\A^u$:
$$
\P(F)=\P(F\cap(\rho_0^u(X)\geq 0))=\P(F\cap(\rho_0^u(X)=0)).
$$
We deduce that $ F\subset(\rho_0^u(X)=0)$. By symmetry the result
follows.
\end{proof}

In the following lemma, we give some properties of the mapping
$v\mapsto \rho_0^v(X)$ for a fixed $X\in\Linf$.
\begin{lem}\label{l3}
For all $X\in\Linf$ and $v,u,w,w^1,...,w^n\in \calN_0$ we
have:
\begin{enumerate}
\item $\rho_0^v(u)\leq \rho_0^v(w)\,\rho_0^w(u)$.
\item $\rho_0^v(X)\leq
\rho_0^v\left(\sum_{k=1}^n\rho_0^{w^k}(X_k)w^k\right)$ whenever
$X=X_1+...+X_n$ with $X_i\in \Linf$ for each $i=1,\ldots,n$.
\item With the convention $\dfrac{0}{0}=0$, we have
$$
\rho_0^{v+w}(X)\leq\dfrac{
\rho_0^v(X)\,\rho_0^w(X)}{\rho_0^v(X)+\rho_0^w(X)}.
$$
\item For $\lambda\in\Linf_+(\F_0)$ such that $\lambda\,v\in\calN_0$,
we have:
$$
\rho_0^{\lambda\,v}(X)=\frac{1}{\lambda}\,\rho_0^v(X).
$$
\item $uv\in \calN_0$ iff
$\hbox{ess-inf}_{\bbQ\in\Q^v}\E_\bbQ(u|\F_0)>0$ a.s and then
$$
\rho_0^{uv}(X)=\left(\rho_0^{(v)}\right)^u(X/v).
$$
\end{enumerate}
\end{lem}
\begin{proof}To prove the assertion $(1)$, we remark that
$$
u-\rho_0^v(w)\,\rho_0^w(u)v=u-\rho_0^w(u)w+\rho_0^w(u)(w-\rho_0^v(w)v)\in\A_0.
$$
To prove $(2)$, we use formula $(ii)$ in Lemma \ref{l2} to see that
for all $\bbQ\in\Q$,
$$
\rho_0^v\left(\sum_{k=1}^n\rho_0^{w^k}(X_k)w^k\right)\geq
\dfrac{1}{\E_\bbQ(v|\F_0)}\E_\bbQ\left(\sum_{k=1}^n\dfrac{\E_{\bbQ}(X_k|\F_0)\,w^k}
{\E_{\bbQ}(w^k|\F_0)}|\F_0\right),
$$
which means that
$$
\rho_0^v\left(\sum_{k=1}^n\rho_0^{w^k}(X_k)w^k\right)\geq
\dfrac{1}{\E_\bbQ(v|\F_0)}\sum_{k=1}^n\E_\bbQ(X_k|\F_0).
$$
So, since $X=X_1+\ldots+X_n$, we have for all $\bbQ\in\Q$:
$$
\rho_0^v\left(\sum_{k=1}^n\rho_0^{w^k}(X_k)w^k\right)\geq
\dfrac{\E_\bbQ(X|\F_0)}{\E_\bbQ(v|\F_0)},
$$
and consequently
$$
\rho_0^v\left(\sum_{k=1}^n\rho_0^{w^k}(X_k)w^k\right)\geq\rho_0^v(X).
$$
To prove $(3)$, we note that $\rho_0^v(X)$ and $\rho_0^w(X)$ have
the same sign and then thanks to Lemma \ref{r1},
$\rho_0^v(X)+\rho_0^w(X)=0$ if and only if $\rho_0(X)=0$. Now
suppose that $\rho_0(X)\neq 0$ and remark that the $\F_0$-measurable
random variable
$$
\alpha(X)\defto \dfrac{\rho_0^v(X)}{\rho_0^v(X)+\rho_0^w(X)},
$$
has values a.s in the interval $[0,1]$, so
$$
X-\dfrac{\rho_0^v(X)\rho_0^w(X)}{\rho_0^v(X)+\rho_0^w(X)}(v+w)
=\alpha(X)(X-\rho_0^w(X)w)+(1-\alpha(X))(X-\rho_0^v(X)v)\in\A_0.
$$
Assertion $(4)$ is an immediate consequence of Lemma \ref{l2}. For
assertion $(5)$ we have for $\bbQ\in\Q$:
$$
\frac{\E_\bbQ(X|\F_0)}{\E_\bbQ(uv|\F_0)}=\frac{\E_\R(X/v|\F_0)}{\E_\R(u|\F_0)},
$$
with
$$
\frac{d\R}{d\bbQ}=\frac{v}{\E_\bbQ(v|\F_0)},
$$
and then $\R\in \Q^v$. We deduce that
$$
\rho_0^{uv}(X)\leq \left(\rho_0^{(v)}\right)^u(X/v).
$$
The reverse inequality is proved in the same way.
\end{proof}

Remark that at time zero, trading between two different num\'eraires
may incur additional costs. Nevertheless the set $\calN_0$ can be
partitioned into equivalence classes so that trade is frictionless
within each one of them.

\begin{defi}
\label{d.4} Two assets $v,w\in \calN_0$ are said to be
$\F_0$-equivalent (or equivalent at time zero) (and we write $v\sim w$) if for
all $X\in
\Linf$, we have $\rho_0^v(X)=\rho_0^v(w)\,\rho_0^w(X)$.
\end{defi}

Some equivalent conditions for $\F_0$-equivalence are given below.
\begin{lem}
\label{l4} Let $v,w\in \calN_0$, then the following are equivalent:
\begin{enumerate}
\item $\rho_0^v(w)\rho_0^w(v)=1$.
\item $w-\rho_0^v(w)v\in lin(\A_0)$ where $lin(\A_0)=\A_0\cap(-\A_0)$ is the
lineality subspace of $\A_0$.
\item $\rho_0^v(-w)=-\rho_0^v(w)$.
\item The two assets $v$ and $w$ are $\F_0$-equivalent.
\item The two assets $v$ and $v+w$ are $\F_0$-equivalent.
\item For all $X\in \Linf$,
$$
\rho_0^{v+w}(X)=\dfrac{\rho_0^v(X)\,\rho_0^w(X)}{\rho_0^v(X)+\rho_0^w(X)},
$$
with the convention $\dfrac{0}{0}=0$.
\end{enumerate}
\end{lem}

\begin{proof}$(1)\Rightarrow( 2)$ We have $v-\rho_0^w(v)w\in\A_0$ by
definition, so
$$
-w+\rho_0^v(w)v=\rho_0^v(w)(v-\rho_0^w(v)w)\in\A_0.
$$
Since $w-\rho_0^v(w)v\in\A_0$, it follows that $w-\rho_0^v(w)v\in
lin(\A_0)$.

$(2)\Rightarrow(1)$ Conversely we have $w-\rho_0^v(w)v\in lin(\A_0)$
which means that $v-\frac{1}{\rho_0^v(w)}w\in \A_0$ and hence
$\rho_0^v(w)\rho_0^w(v)\leq 1$. From assertion $(1)$ in Lemma
\ref{l3} we deduce that $\rho_0^v(w)\rho_0^w(v)=1$.

$(2)\Rightarrow( 3)$ We have $\rho_0^v(-w+\rho_0^v(w)v)=0$ since
$-w+\rho_0^v(w)v\in lin(\A_0)$ and then
$0=\rho_0^v(-w+\rho_0^v(w)v)=\rho_0^v(-w)+\rho_0^v(w)$.

$(3)\Rightarrow( 2)$ Conversely we have, by definition,
$-w-\rho_0^v(-w)v\in\A_0$, so
$$
-w+\rho_0^v(w)v=-w-\rho_0^v(-w)v\in\A_0.
$$

$(1)\Leftrightarrow( 4)$ to prove the forward implication, note that
$$
X-\rho_0^v(w)\rho_0^w(X)v=(X-\rho_0^w(X)w)+\rho_0^w(X)(w-\rho_0^v(w)v)\in\A_0,
$$
since $X-\rho_0^w(X)w\in\A_0$ and $w-\rho_0^v(w)v\in lin(\A_0)$.
Hence $ \rho_0^v(X)\leq \rho_0^v(w)\rho_0^w(X)$. Swapping the roles
of $v$ and $w$ we obtain also that $ \rho_0^w(X)\leq
\rho_0^w(v)\rho_0^v(X)$ and then
$$
\rho_0^v(X)\leq
\rho_0^v(w)\rho_0^w(X)\leq\rho_0^v(w)\rho_0^w(v)\rho_0^v(X)=\rho_0^v(X).
$$
The converse is trivial.

$(3)\Rightarrow( 5)$ we have
$$
\rho_0^v(-(v+w))=\rho_0^v(-v-w)=\rho_0^v(-w)-1=-\rho_0^v(w)-1=-\rho_0^v(v+w).
$$
$(5)\Rightarrow(3)$ Conversely we have
$\rho_0^v(-w)=\rho_0^v(-v-w)+1=-\rho_0^v(v+w)+1=-\rho_0^v(w)$.

$(1)\Rightarrow(6)$ Assume $(1)$ holds, then, as we have previously
proved, $(4)$ and $(5)$ hold. Then
$$
\rho_0^{v+w}(X)=\rho_0^{v+w}(v)\rho_0^v(X)=
\dfrac{\rho_0^v(X)}{\rho_0^v(v+w)}=
\dfrac{\rho_0^v(X)}{1+\rho_0^v(w)}=
\dfrac{\rho_0^v(X)\rho_0^w(X)}{\rho_0^v(X)+\rho_0^w(X)}.
$$
$(6)\Rightarrow(1)$ Conversely, we take $X=v+w$ and deduce assertion
$(1)$.
\end{proof}

\begin{remark}
\label{r3} Assertion $(2)$ in Lemma \ref{l4} is equivalent
to
\begin{itemize}
\item[($2'$)]: $w-\al v\in lin(\A_0)$ for some $\al\in\Linf_+(\F_0)$ such
that $1/\al\in\Linf$.
\end{itemize}
Indeed for the direct implication we can take $\al=\rho_0^v(w)$. For
the converse, we have
$$
\rho_0^w(v)\leq \dfrac{1}{\al}\leq\dfrac{1}{\rho_0^v(w)},
$$
and then $\rho_0^v(w)\rho_0^w(v)=1$. This implies $(2)$.
\end{remark}

\begin{cor}
\label{l5}The binary relation $\sim$ defined on $\calN_0$ in
Definition \ref{d.4}, is an equivalence relation. For all $v\in
\calN_0$, the subset $[v]\cup\{0\}$ is a convex cone where
$[v]\defto\{w\in \calN_0: w\sim v\}$ denotes the equivalence class
of $v$. Moreover
$$
\overline{[v]}=\left\{x\in \overline{\calN_0}:\;\exists \lambda\in
\Linf_+(\F_0)\;\hbox{such that}\; x-\lambda v\in
lin(\A_0)\right\}\defto {\mathcal E}(v),
$$
where the closure is taken with respect to the weak$^*$ topology in
$\Linf$.
\end{cor}
\begin{proof}From Remark \ref{r3} we deduce that $\sim$ is an equivalence
relation and
from the equivalence of properties $(4)$ and $(5)$ in Lemma
\ref{l4}, we deduce easily that for each $v\in \calN_0$, the subset
$[v]\cup\{0\}$ is a convex cone. Now we prove that
$\overline{[v]}\subset {\mathcal E}(v)$. Let $w^n$ be a sequence in
$[v]$ which converges weakly$^*$ to $w$ in $\Linf$. The sequence
$$
\alpha^n\defto
\rho_0^v(w^n)=\dfrac{\E_\bbQ(w^n|\F_0)}{\E_\bbQ(v|\F_0)},
$$
for a fixed $\bbQ\in \Q^e$, converges weakly$^*$ to
$$
\alpha\defto \dfrac{\E_\bbQ(w|\F_0)}{\E_\bbQ(v|\F_0)}.
$$
By working with $w^n+\vare\,v$ and then taking the limit in $\vare$,
we suppose, without loss of generality, that $\al\geq \vare>0$. So
the sequence $1/\al^n$ is bounded, thus there exists an
$\F_0$-measurable integer-valued sequence $\tau_n$ such that the
sequence $\al^{\tau_n}$ converges a.s to some $\alpha$ and then
$\al\in\Linf$ and $1/\al\in\Linf$. We know that for all $\bbQ\in\Q$
and $f\in \L^1(\F_0)$ we have
\begin{eqnarray}
\label{number}
\E_\bbQ(f\,w^{\tau_n})=\sum_{k}\E_\bbQ(f\,1_{(\tau_n=k)}w^{k})
=\sum_{k}\E_\bbQ(f\,1_{(\tau_n=k)}\al^{k}\,v)=\E_\bbQ(f\,\alpha^{\tau_n}\,v),
\end{eqnarray}
since for all $k$, $w^k\sim v$ and $f\,1_{(\tau_n=k)}\in
\L^1(\F_0)$. So the left hand side of (\ref{number}) converges to
$\E_{\bbQ}(f\,w)$ and the right hand side converges to
$\E_\bbQ(f\,\alpha\,v)$. Hence $ \E_\bbQ(w-\alpha\,v|\F_0)=0$ for
all $\bbQ\in\Q$ and so $w-\alpha\,v\in lin(\A_0)$.

Conversely, let $w\in {\mathcal E}(v)$ which means that $w=\al\,v+z$
with $z\in lin(\A_0)$. Define
$$
w^n\defto w+\frac{1}{n}v=\left(\al+\frac{1}{n}\right)\,v+z \in
\calN_0,
$$
then $w^n\in [v]$ from Remark \ref{r3} and $w^n$ converges
weakly$^*$ to $w$ in $\Linf$.
\end{proof}

\section{proofs of results in section \ref{abscone}}\label{app3}
{\em Proof of Lemma \ref{duality0}:\ }
The inclusion $D\subset\cap_{t=0}^{T-\eta}M^\eta_t(D)$ is
trivial.

Now we let $Y\in \cap_{t=0}^{T-\eta}M^\eta_t(D)$ and seek to prove that $Y\in
D$.
So, for all $t\in \{0,\ldots,T-\eta\}$, there exists some $Z^t\in
D$, $\be_t\in \L^0_+(\G_{t})$ with $\beta_t Z^t\in \L^1$ such that
$Z_{t+\eta}=\be_t\,Z^t_{t+\eta}$. Define $\xi^{T-\eta}=Z^{T-\eta}$
and for $t\in\{0,\ldots,T-\eta-1\}$:
$$
\xi^t=1_{F_t}\k_t\,\xi^{t+1}+1_{F^c_t}\,Z^{t},
$$
where $F_t=(\be_t>0)$ and $\k_t=\be_{t+1}/\be_{t}$. Remark that
$$
Z_{t+1+\eta}=\be_{t+1}\,Z^{t+1}_{t+1+\eta},
$$
and
$$
Z_{t+\eta}=\be_{t}\,Z^{t}_{t+\eta}.
$$
Thus
$$
\E[\be_{t+1}Z^{t+1}|\G_{t+\eta}]=\E[Z_{t+1+\eta}|\G_{t+\eta}]=Z_{t+\eta}
=\be_{t}\,Z^{t}_{t+\eta},
$$
which leads us to deduce that
$$
\E((1_{F_t}\k_t\,Z^{t+1}+1_{F^c_t}Z^{t})|\G_{t+\eta})=Z^{t}_{t+\eta}.
$$
Remark also that, since $D\subset\L^1_+(\G;\R^d)$,
$$
Z=\be_0\,\k_0\times\ldots\times\k_{T-\eta-1}\,Z^{T-\eta}.
$$
We deduce that $Z=\be_0\,\xi^0$.

Now we prove by backwards induction
on $t$ that $Z^{t}_{t+\eta}=\xi^{t}_{t+\eta}$ and $\xi^t\in D$. For
$t=T-\eta$, we have $\xi^{T-\eta}=Z^{T-\eta}\in D$ by definition.
Suppose that for all $s=T-\eta,\ldots,t+1$, we have
$Z^{s}_{s+\eta}=\xi^{s}_{s+\eta}$ and $\xi^s\in D$ and show that
$Z^{t}_{t+\eta}=\xi^{t}_{t+\eta}$ and $\xi^t\in D$. First
$\xi^t=1_{F_t}\k_t\,\xi^{t+1}+1_{F^c_t}\,Z^{t}$ so we get
$$
\xi^{t}_{t+\eta}=\E((1_{F_t}\k_t\,\xi^{t+1}+1_{F^c_t}\,Z^{t})|\G_{t+\eta})
=\E((1_{F_t}\k_t\,Z^{t+1}+1_{F^c_t}\,Z^{t})|\G_{t+\eta})=Z^{t}_{t+\eta}.
$$
Then
$$
Z^{t}_{t+\eta}=\E((1_{F_t}\k_t\,\xi^{t+1}+1_{F^c_t}\,Z^{t})|\G_{t+\eta}),
$$
with $Z^t,\xi^{t+1}\in D$. By the $\eta$-stability of the subset $D$
we deduce that $\xi^t\in D$ and consequently $Z=\be_0\,\xi^0\in D$
since $\xi^0\in D$ and $\be_0$ is a positive scalar.
\endpf

{\em Proof of Lemma \ref{duality1}:\ }
\begin{enumerate}
\item We remark that $[D]$ is a closed convex cone in $\L^1$,
containing $D$. Now we prove that $[D]$ is $\eta$-stable. Fix
$t\in\{0,1,\ldots,T-\eta\}$ and suppose that $Z^1,\ldots,Z^{k}\in [D]$ are such
that there
exists $Z\in [D]$, a partition $F^1_t,\ldots,F^{k}_t\in\F_t$ and
$\al^1,\ldots,\al^{k}\in\L^0(\F_{t-\eta+1})$ with each
$\al^i\,Z^i\in \L^1$ and
$$
\E(Z|\,\G_{t})=\E(\al^i\,Z^i|\G_t),
$$
on $F^i_t$ for each $i=1,\ldots,k$. We want to prove that
$$
Y\defto\sum_{i=1}^{k}1_{F^i_t} \al^i\,Z^i\in [D].
$$
Now for $s\geq t-\eta+1$ we have
$$
Y_{s+\eta}=\sum_{i=1}^{k}1_{F^i_t} \al^i\,Z^i_{s+\eta},
$$
with $Z^i_{s+\eta}=W^i_{s+\eta}$ for some $W^i\in D_{(s)}$.
Therefore
$$
Y_{s+\eta}=\left(\sum_{i=1}^{k}1_{F^i_t} \al^i\,W^i\right)_{s+\eta},
$$
with $\sum_{i=1}^{k}1_{F^i_t} \al^i\,W^i\in D_{(s)}$.

Now
for $s\leq t-\eta$ we have $Y_{s+\eta}=(Y_{t})_{s+\eta}$ and
$$
Y_t=\sum_{i=1}^{k}1_{F^i_t} \E(\al^i\,Z^i|\G_{t})=Z_t,
$$
and then $Y_{s+\eta}=Z_{s+\eta}=W_{s+\eta}$ for some $W\in D_{(s)}$.
Therefore $Y\in [D]$.

\item We prove that $D=[D]\,\Leftrightarrow\, D$ is an $\eta$-stable closed
convex cone
in $\L^1$. For the reverse implication, thanks to Lemma
\ref{duality0} we have:
$$
D=D^{**}=\cap_{t=0}^{T-\eta}\overline{conv}(M^\eta_t(D))=\cap_{t=0}^{T-\eta}\overline{R^\eta_t},
$$
with $D\subset[D]\subset \cap_{t=0}^{T-\eta}\overline{R^\eta_t}$.
Then $D=[D]$. The direct implication is trivial from the first
assertion.

\item To prove that $[D]$ is the smallest $\eta$-stable closed convex cone in
$\L^1$ which contains $D$, simply let $D'$ be an $\eta$-stable closed convex
cone in $\L^1$, containing $D$. Then $[D]\subset [D']=D'$.
\end{enumerate}

\endpf

{\em Proof of Theorem \ref{rafa5}:\ }
Remark that $\B^*$ is a closed convex cone in $\L^1$. We claim that if we can
prove that
for all $t=0,\ldots,T-\eta$ we have
$K_t^\eta(\B)=\left(M_t^\eta(\B^*)\right)^*$, the result will follow
thanks to Lemmas \ref{duality0} and \ref{duality1}. To see this, note
first
that Lemma \ref{duality0} tells us that $\B^*$ being $\eta$-stable
implies that $\B^*=\cap M^\eta_t$. Thus if $\K^\eta_t=(M^\eta_t)^*$ then
$\B=\B^{**}=(\cap
M^\eta_t)^*=\overline{\oplus(M^\eta_t)^*}=\overline{\oplus(\K^\eta_t)}$,
establishing the reverse implication. Conversely,
$(M^\eta_t)^*=\left(\overline{conv}(M^\eta_t)\right)^*=(R^\eta_t)^*$, so, if
$\B=\overline{\oplus(\K^\eta_t)}=\overline{\oplus(M^\eta_t)^*}$, then
$\B=\overline{\oplus(R^\eta_t)^*}$ so $\B^*=\cap R^\eta_t=[\B^*]$ and then,
by Lemma \ref{duality1}, $\B^*$ is $\eta$-stable.

First we prove that $M^\eta_t(\B^*)\subset (\K^\eta_t(\B))^*$. Let
$Z\in M^\eta_t(\B^*)$, then there exists some $Z'\in \B^*$ and
$\al\in\L^0_+(\F_t)$ with $\al\,Z'\in\L^1$ such that
$Z_{t+\eta}=\al\,Z'_{t+\eta}$. Take $X\in \K^\eta_t(\B)$, then
$$
\E(Z.X)=\E(Z_{t+\eta}.X)=\E(\al_t\,Z'_{t+\eta}.X),
$$
since $X\in \Linf(\G_{t+\eta},\R^d)$. We then obtain:
$$
\E(Z.X)=\E(Z'_{t+\eta}.\al_t\,X)=\lim_{n\to\infty}\E(Z'.\al_t\,1_{(\al_t\leq
n)}\,X)\leq 0,
$$
since $\al_t\,1_{(\al_t\leq n)}\,X\in \B$ and $Z'\in \B^*$.

Now we
prove that $(M^\eta_t(\B^*))^*\subset \K^\eta_t(\B)$. We remark that
$\B^*\subset M^\eta_t(\B^*)$ and also that
$\Linf_+(\G_t)\,M^\eta_t(\B^*)\subset M^\eta_t(\B^*)$, so
$(M^\eta_t(\B^*))^*\subset \B_t$. Let $X\in (M^\eta_t(\B^*))^*$, we
want to prove that $X\in \Linf(\G_{t+\eta},\R^d)$. Let $Z\in
\L^1(\F,\R^d)$, we remark that $Z-Z_{t+\eta}\in M^\eta_t(\B^*)$ and
consequently $\E((Z-Z_{t+\eta}).X)\leq 0$. We deduce then that
$\E((X-X_{t+\eta}).Z)=\E((Z-Z_{t+\eta}).X)\leq 0$ for all $Z\in
\L^1$. Therefore $X=X_{t+\eta}$.

\endpf

{\em Proof of Lemma \ref{lem5.3}:\ }
 $(1)\Leftrightarrow(2)$. For $Z\in \B^*$ and $f^+_t\in \Linf_+(\G_t)$ we have:
$$
\E f^+_t(Z_t.X)=\E Z.(f^+_t\,X)\leq 0,
$$
since $f^+_t\,X\in \B$. Then $Z_t.X\leq 0$ a.s. Conversely let
$f^+_t\in \Linf_+(\G_t)$ we want to prove that $f^+_t\,X\in \B$. Let
$Z\in \B^*$ then
$$
\E Z.(f^+_t\,X)=\E Z_t.(f^+_t\,X)=\E f^+_t(Z_t.X)\leq 0.
$$
Therefore $f^+_t\,X\in \B$.

$(2)\Leftrightarrow(3)$, for all $W\in \L^{1}$ such that $W_t=
Z_t$ with $Z\in \B^*$ and $f^+_t\in \Linf_+(\G_t)$, we have
$$
\E(f^+_t\,(W.X))=\E(f^+_t\,W_t.X)=\E(f^+_t\,Z_t.X)\leq 0.
$$
Conversely we prove first that $X\in \Linf(\G_t)$. Remark that for
every $W\in \L^{1}$ we have $\E(W-W_t|\,\G_t)= 0$ with $0\in \B^*$,
then $\E[(W-W_t).X]\leq 0$. Consequently for every $W\in \L^{1}$ we
get
$$
\E W.(X-X_t)=\E(W-W_t).X\leq 0,
$$
and so $\E W.(X-X_t)=0$ for every $W\in \L^{1}$. Then
$X=X_t\defto\E(X|\,\G_t)$. Let $Z\in \B^*$, then
$$
Z_t.X=\E(Z.X|\,\G_t)\leq 0.
$$
\endpf



\section{Further results on $\eta$-decomposability}\label{app2}
\begin{prop}
\label{p1}We have:
\begin{enumerate}
\item For all $t\in\{0,\ldots,T\}$,
$(\B_t)^*=\overline{(\B^*)_{t,T}}\defto(\B^*)_{(t)}$, where
$$
(\B^*)_{t,T}=\{\al Z;\; \al\in \L^\infty(\F_t),\,Z\in \B^*\}.
$$
\item Define $\B^{\eta}\defto\overline{\oplus_{t=0}^{T-\eta}\,\K^\eta_t(\B)}$;
then
$\B^{\eta}$ is the largest $\eta$-decomposable closed convex cone in
$\B$.
\end{enumerate}
\end{prop}
\begin{proof}To prove $(1)$, we fix $t\in\{0,\ldots,T\}$,\,$Z\in (\B^*)_{t,T}$
and
$X\in \B_t$, then there exists some $Z'\in \B^*$ and $\al\in
\Linf_+(\G_t)$ such that $Z=\al\,Z'$ and
$$
\E Z.X=\E \al\,Z'.X=\E Z'.(\al\,X)\leq 0.
$$
We deduce that $(\B^*)_{t,T}\subset (\B_t)^*$.

Now let $X\in \Linf$ be such that $\E Z.X\leq 0$ for all $Z\in
(\B^*)_{t,T}$. Then, in particular for all $Z\in \B^*$ and $\al\in
\Linf_+(\G_t)$ we have
$$
\E Z.(\al\,X)=\E (\al\,Z).X\leq 0,
$$
since $\al\,Z\in (\B^*)_{t,T}$, which implies that $\al X\in\B$ for
all $\al\in \Linf_+(\G_t)$ and then $X\in\B_t$.

To prove $(2)$, remark that for all $t\in\{0,\ldots,T-\eta\}$, we have
$\K^\eta_t(\B)\subset \B^{\eta}\subset \B$, then
$\K^\eta_t(\B^{\eta})=\K^\eta_t(\B)$ and so
$\B^{\eta}=\overline{\oplus_{t=0}^{T-\eta}\,\K^\eta_t(\B^{\eta})}$.
Now let $M$ be an $\eta$-decomposable closed convex cone in $\B$,
then for all $t\in\{0,\ldots,T-\eta\}$ we have $\K^\eta_t(M)\subset
\K^\eta_t(\B)\subset \B^{\eta}$. We deduce that
$M=\overline{\oplus_{t=0}^{T-\eta}\K^\eta_t(M)}\subset \B^{\eta}$.
\end{proof}

\goodbreak

\begin{cor}
\label{cor1} We have
\begin{enumerate}
\item If $\B$ is $\eta$-decomposable, then each $\B_t$ is $\eta$-decomposable.
\item For fixed $t\in\{0,\ldots,T\}$, we have $\B^{\eta}_t\defto
(\B^\eta)_t=(\B_t)^{\eta}$.
\item For fixed $t\in\{0,\ldots,T\}$, $\B_t$ is $\eta$-decomposable if and only
if
$\B_t=\B^{\eta}_t$.
\end{enumerate}
\end{cor}

\begin{proof}Suppose that $\B$ is $\eta$-decomposable, then $\B^*$ is
$\eta$-stable.
Hence for all $t\in\{0,\ldots,T\}$, $(\B^*)_{t,T}$ is $\eta$-stable:
indeed for $s\in\{0,\ldots,T-\eta\}$, consider $Z,Z^1,\ldots,Z^k\in
(\B^*)_{t,T}$ and a partition $F^1_s,\ldots,F^k_s\in\F_s$ and
$\al^1,\ldots,\al^k\in \L^0_+(\G_{s-\eta+1})$ with $\al^i Z^i\in
\L^1$ such that $Y\defto \sum_{i=1}^k 1_{F^i_t}\al^i Z^i$ satisfies:
$$
Z_s=\E(Y|\G_s).
$$
For $s\leq t+\eta-1$, we have $Y\in (\B^*)_{t,T}$ and for $s\geq
t-\eta$, by definition of $(\B^*)_{t,T}$, there exists
$\beta,\beta^1,\ldots,\beta^k\in \Linf_+(\G_t)$ and
$W,W^1,\ldots,W^k\in \B^*$ such that $Z=\beta W$ and $Z^i=\beta^i
W^i$. Then
$$
\beta W_s=\sum_{i=1}^k 1_{F^i_t}\beta^i \E(\al^i W^i|\G_s),
$$
which means that
$$
W_s=\E\left(1_G \dfrac{\sum_{i=1}^k 1_{F^i_t}\beta^i \al^i
W^i}{\beta}+1_{G^c} W|\G_s\right),
$$
where $G=(\beta>0)$. Since $\B^*$ is $\eta$-stable, it follows that
$$
Y'\defto 1_G Y+1_{G^c} W\in \B^*,
$$
and consequently $Y=\beta Y' \in \B^*_{t,T}$.

By Theorem \ref{rafa5}, $\B_t$ is $\eta$-decomposable. Now to prove
$(2)$, remark that $\B^\eta_t$ is $\eta$-decomposable (this follows
by assertion $(1)$ since $\B^\eta$ is $\eta$-decomposable) and for
all $s\in\{t+1,\ldots,T-\eta\}$ we have
$$
\K^\eta_s(\B^{\eta}_t)=\K^\eta_s(\B^{\eta})=\K^\eta_s(\B)=\K^\eta_s(\B_t),
$$
and for $s\leq t$, we have $\K^\eta_s(\B_t)\subset \K^\eta_t(\B_t)$.
Hence
$$
\B^{\eta}_t=\overline{\oplus_{s=0}^{T-\eta}\K^\eta_s(\B^{\eta}_t)}
=\overline{\oplus_{s=t}^{T-\eta}\K^\eta_s(\B_t)}=(\B_t)^\eta.
$$
To prove $(3)$, assume that $\B_t$ is $\eta$-decomposable, then
$\B_t=(\B_t)^\eta=\B^\eta_t$. The converse is obvious.
\end{proof}

\end{appendix}


\begin{thebibliography}{99}
\bibitem{ADEH} Artzner, Delbaen, Eber and Heath $(1999)$: Coherent measures of
risk.
{\it Math. Finance} $9$ no. $3$, $203-228$.
\bibitem{art}P. Artzner, F. Delbaen, J.M. Eber, D. Heath, H. Ku $(2004)$,
Coherent multiperiod
risk adjusted values and Bellman's principle. Working Paper, ETH
Z¨urich.
\bibitem{cher}P. Cheridito, F. Delbaen, M. Kupper $(2006)$, Dynamic monetary
risk measures
for bounded discrete-time
processes. {\it Electronic Journal of Probability}, Vol. 11, Paper no. 3,
pages 57-106.
\bibitem{det} K. Detlefsen and G. Scandolo $(2005)$, Conditional and
dynamic convex risk measures. {\it Fin. \& Stochastics} 9, 539-561.
\bibitem{eng}J. Engwerda, B. Roorda, H. Schumacher $(2004)$, Coherent
acceptability measures
in multiperiod models. Working Paper, University of Twente, Math.
Finance.
\bibitem{delbaen}F. Delbaen $(2002)$, Coherent risk measures on general
probability spaces. {\it Advances in Finance and Stochastics,
$1-37$, Springer, Berlin}.
\bibitem{delbaen2}F. Delbaen, The structure of m-stable sets and
in particular the set of risk neutral measures. Preprint
(http://www.math.ethz.ch/~delbaen/)
\bibitem{DKV} F. Delbaen, Y. M. Kabanov and E. Valkeila: Hedging under
transaction costs in currency markets: a discrete-time model. {\it Math.
Finance} 12 (2002),
no. 1, 45--61.
\bibitem{jacka}S. D. Jacka $(1992)$, A martingale representation result and an
application to incomplete financial markets.  {\it Math. Finance}
$2$, $23-34$.
\bibitem{JB}S. D. Jacka and A. Berkaoui $(2006)$, On decomposing risk in a
financial-intermediate market and reserving.
http://arxiv.org/abs/math.PR/0603041
\bibitem{JBW}S. D. Jacka, A. Berkaoui and J. Warren $(2006)$, No arbitrage and
closure results for trading cones with transaction costs.
http://arxiv.org/abs/math.PR/0602178
\bibitem{JK} E. Jouini and H. Kallal: Arbitrage in securities markets with
short-sales constraints. {\it Math. Finance} 5 (1995), no. 3, 197--232.
\bibitem{kabanov}Yu. M. Kabanov $(1998)$:
Hedging and liquidation under transaction costs in currency markets,
{\it Fin. \& Stochastics} 3(2), 237--248.
\bibitem{kabanov2}Yu. M. Kabanov,  M. Rasonyi and Ch. Stricker $(2002)$:
No-arbitrage criteria for financial
markets with efficient friction, {\it Fin. \& Stochastics} 6(3), 371--382.
\bibitem{kabanov3}Yu. M. Kabanov,  M. Rasonyi and Ch. Stricker $(2003)$: On the
closedness of
sums of convex cones in $\L^0$ and the robust no-arbitrage property,
{\it Fin. \& Stochastics} 7(3), 403--411.
\bibitem{riedel}F. Riedel $(2004)$, Dynamic conditional coherent
risk measures. {\it Stochast. Proc. Appl.} 112, 185--200.
\bibitem{scan}G. Scandolo $(2003)$, Risk measures in a dynamic setting. PhD
Thesis, Universit\`a di
Milano.
\bibitem{schacher} W. Schachermayer $(2004)$, The fundamental theorem of asset
pricing under proportional transaction costs
in discrete time.  {\it Math. Finance} $14$ no. $1$, $19-48$.
\bibitem{wang}T. Wang $(1999)$, A class of dynamic risk measures. Working Paper,
University of British Columbia.
\bibitem{weber}S. Weber $(2003)$, Distribution-invariant dynamic risk measures.
Working Paper, HU
Berlin.




\end{thebibliography}
\end{document}